\documentclass[a4paper]{amsart}
\usepackage{color}
\usepackage[latin1]{inputenc}
\usepackage[T1]{fontenc}
\usepackage{amsfonts}
\usepackage{amssymb}
\usepackage{amsmath}
\usepackage{amsthm}
\usepackage{stmaryrd}
\usepackage{eufrak}
\usepackage{graphicx,type1cm,eso-pic,color}
\usepackage{pstricks,pst-plot,pstricks-add}
\usepackage{float}
\newtheorem{theorem}{Theorem}[section]
\newtheorem{lemma}[theorem]{Lemma}
\newtheorem{proposition}[theorem]{Proposition}
\newtheorem{corollary}[theorem]{Corollary}
\newtheorem{definition}[theorem]{Definition}
\newtheorem{assumption}[theorem]{Assumption}
\newtheorem{remark}[theorem]{Remark}
\begin{document}
\setlength\arraycolsep{2pt}
\title{Invariance for Rough Differential Equations}
\author{Laure COUTIN*}
\author{Nicolas MARIE**}
\address{*Institut de math\'ematiques de Toulouse, Universit\'e Paul Sabatier, Toulouse, France}
\email{laure.coutin@math.univ-toulouse.fr}
\address{**Laboratoire Modal'X, Universit\'e Paris 10, Nanterre, France}
\email{nmarie@u-paris10.fr}
\address{**Laboratoire ISTI, ESME Sudria, Paris, France}
\email{marie@esme.fr}
\date{}
\maketitle
%


%
\begin{abstract}
In 1990, in It\^o's stochastic calculus framework, Aubin and Da Prato established a necessary and sufficient condition of invariance of a nonempty compact or convex subset $C$ of $\mathbb R^d$ ($d\in\mathbb N^*$) for stochastic differential equations (SDE) driven by a Brownian motion. In Lyons rough paths framework, this paper deals with an extension of Aubin and Da Prato's results to rough differential equations. A comparison theorem is provided, and the special case of differential equations driven by a fractional Brownian motion is detailed.
\end{abstract}
\tableofcontents
%


%
\section{Introduction}
The invariance of a nonempty closed convex subset of $\mathbb R^d$ ($d\in\mathbb N^*$) for a
(ordinary) differential equation was solved by Nagumo in \cite{NAGUMO42}, see also \cite{AUBIN90} for a simple proof. It was obtained by Aubin and Da Prato in \cite{AD90} for stochastic differential equations. More explicit results in the special case of polyhedrons have been established in Milian \cite{MILIAN95}. In \cite{CPS13}, Cresson, Puig and Sonner have introduced a stochastic generalization of the well-known Hodgkin-Huxley neuron model satisfying the assumptions of the stochastic viability theorem of Milian \cite{MILIAN95}. On the viability and the invariance of sets for stochastic differential equations, see also Milian \cite{MILIAN93}, Gautier and Thibault \cite{GT93}, and Michta \cite{MICHTA98}. 
\\
\\
In \cite{AD98}, the results of \cite{AD90} were extended by Aubin and Da Prato to the stochastic differential inclusions. The case of stochastic controlled differential equations was studied by Da Prato and Frankowska in \cite{PF01} or more recently by Buckdahn, Quicampoix, Rainer and Teichmann in \cite{BQRT10}. An unified approach which provides a viability theorem for stochastic differential equations, backward stochastic differential equations and partial differential equations is developed in Buckdahn et al. \cite{BQRR02}.
\\
\\
The invariance of a subset of $\mathbb R^d$ for a stochastic differential equation driven by a $\alpha$-H\"older continuous process with $\alpha\in (1/2,1)$ has been already studied by several authors in the fractional calculus framework developed by Nualart and Rascanu in \cite{NR02}. In Ciotir and Rascanu \cite{CR08} and Nie and Rascanu \cite{NR11}, the authors have proved a sufficient and necessary condition for the invariance of a closed subset of $\mathbb R^d$ for a stochastic differential equation driven by a fractional Brownian motion of Hurst parameter $H\in (1/2,1)$. In \cite{MMS15}, Melnikov, Mishura and Shevchenko have proved a sufficient condition for the invariance of a smooth and nonempty subset of $\mathbb R^d$ for a stochastic differential equation driven by a mixed process containing both a Brownian motion and a $\alpha$-H\"older continuous process with $\alpha\in (1/2,1)$.
\\
\\
The rough paths theory introduced by T. Lyons in 1998 in the seminal paper \cite{LYONS98} provides a natural and powerful framework to study differential equations driven by $\alpha$-H\"older signals with $\alpha\in (0,1]$. The theory and its applications are widely studied by many authors. For instance, see the book of Friz and Victoir \cite{FV10}, the nice introduction of Friz and Hairer \cite{FH14}, or the approach of Gubinelli \cite{GUBINELLI04}.
\\
\\
The main purpose of this article is to extend the viability theorem of Aubin and Da Prato \cite{AD90} and to provide a comparison theorem for the rough differential equations. The paper deals also with an application of the viability theorem to stochastic differential equations driven by a fractional Brownian motion of Hurst parameter greater than $1/4$.
\\
\\
Let $T > 0$ be arbitrarily chosen, and consider the differential equation
\begin{equation}\label{main_equation}
y_t = y_0 +
\int_{0}^{t}b(y_s)ds +
\int_{0}^{t}\sigma(y_s)dw_s
\textrm{ ; }t\in [0,T]
\end{equation}
where, $y_0\in\mathbb R^d$, $b$ (resp. $\sigma$) is a continuous map from $\mathbb R^d$ into itself (resp. $\mathcal M_{d,e}(\mathbb R)$), and $w : [0,T]\rightarrow\mathbb R^e$ is a $\alpha$-H\"older continuous signal with $e\in\mathbb N^*$ and $\alpha\in (0,1]$.
\\
\\
At Section 2, some definitions and results on rough differential equations are stated in order to take Equation (\ref{main_equation}) in that sense. Section 3 deals with a viability theorem for Equation (\ref{main_equation}) taken in the sense of rough paths (see Friz and Victoir \cite{FV10}) and a convex or compact set. At Section 4, a comparison theorem for the rough differential equations is proved by using the viability results of Section 3. At Section 5, the viability theorem is applied to stochastic differential equations driven by a fractional Brownian motion of Hurst parameter greater than $1/4$. Finally, Appendix A is a brief survey on convex analysis.
\\
\\
For the sake of readability, all results are proved on $[0,T]$, but they can be extended on $\mathbb R_+$ via some usual localization arguments.
\\
\\
The results established in this paper could be applied in stochastic analysis itself, and in other sciences as neurology. On the one hand, in stochastic analysis, one could study the viability of rough differential inclusions as in Aubin and Da Prato \cite{AD98} in It\^o's calculus framework, or could also compare the viability condition for rough differential equations to the reflecting boundary conditions for Ito's stochastic differential equations (see Lions and Sznitman \cite{LS84}). On the other hand, together with J.M. Guglielmi who is neurologist at the American Hospital of Paris, we are studying a fractional Hodgkin-Huxley neuron model, that extends the model of Cresson et al. \cite{CPS13}, in order to model injured nerves membrane potential in some neuropathies.
\\
\\
The following notations are used throughout the paper.
\\
\\
\textbf{Notations (general) :}
\begin{itemize}
 \item The Euclidean scalar product on $\mathbb R^d$ is denoted by $\langle .,.\rangle$, and the Euclidean norm on $\mathbb R^d$ is denoted by $\|.\|$. The canonical basis of $\mathbb R^d$ is denoted by $(e_k)_{k\in\llbracket 1,d\rrbracket}$. For every $x\in\mathbb R^d$, its $j$-th coordinate with respect to $(e_k)_{k\in\llbracket 1,d\rrbracket}$ is denoted by $x^{(j)}$ for every $j\in\llbracket 1,d\rrbracket$.
 \item For every $x_0\in\mathbb R^d$ and $r\in\mathbb R_+$, $B_d(x_0,r) :=\{x\in\mathbb R^d :\|x - x_0\|\leqslant r\}$.
 \item The interior, the closure and the frontier of a set $S\subset\mathbb R^d$ are respectively denoted by $\textrm{int}(S)$, $\overline S$ and $\partial S$.
 \item For every $k\in\llbracket 1,d\rrbracket$, $D_k :=\{x\in\mathbb R^d : x^{(k)}\geqslant 0\}$.
 \item For a nonempty closed set $S\subset\mathbb R^d$, and every $x\in\mathbb R^d$, $\Pi_K(x)$ denotes the set of best approximations of $x$ by the elements of $K$ :
 \begin{equation}\label{set_of_projections}
 \Pi_S(x) :=
 \left\{x^*\in S :\|x - x^*\| =\inf_{y\in S}\|x - z\|\right\}.
 \end{equation}
 \item The distance between $x\in\mathbb R^m$ and a nonempty closed set $S\subset\mathbb R^d$ is
 \begin{displaymath}
 d_S(x) :=
 \inf_{y\in S}\|x - y\|.
 \end{displaymath}
 \item The space of the matrices of size $d\times e$ is denoted by $\mathcal M_{d,e}(\mathbb R)$. The (euclidean) matrix norm on $\mathcal M_{d,e}(\mathbb R)$ is denoted by $\|.\|_{\mathcal M_{d,e}(\mathbb R)}$. If $d = e$, then $\mathcal M_d(\mathbb R) :=\mathcal M_{d,e}(\mathbb R)$. The canonical basis of $\mathcal M_{d,e}(\mathbb R)$ is denoted by $(e_{k,l})_{(k,l)\in\llbracket 1,d\rrbracket\times\llbracket 1,e\rrbracket}$.
 \item Let $E$ and $F$ be two vector spaces. The space of the linear maps from $E$ into $F$ is denoted by $\mathcal L(E,F)$. If $E = F$, then $\mathcal L(E) :=\mathcal L(E,F)$.
 \item The space of the continuous functions from $[0,T]$ into $\mathbb R^d$ is denoted by $C^0([0,T],\mathbb R^d)$ and equipped with the uniform norm $\|.\|_{\infty,T}$.
 \item The space of the continuous functions $l$ from $(0,t_0)$ into $]0,\infty[$ with $t_0 > 0$, and such that
 \begin{displaymath}
 \lim_{t\rightarrow 0^+}\frac{t^{\beta}}{l(t)} = 0
 \textrm{ $;$ }
 \forall\beta > 0,
 \end{displaymath}
 is denoted by $\mathcal S_{t_0}$.
\end{itemize}
\textbf{Notations (rough paths).} See Friz and Victoir \cite{FV10}, Chapters 5, 7, 8 and 9 :
\begin{itemize}
 \item $\Delta_T :=\{(s,t)\in\mathbb R_{+}^{2} : 0\leqslant s < t\leqslant T\}$.
 \item The space of the $\alpha$-H\"older continuous functions from $[0,T]$ into $\mathbb R^d$ is denoted by $C^{\alpha\textrm{-H\"ol}}([0,T],\mathbb R^d)$ and equipped with the $\alpha$-H\"older semi-norm $\|.\|_{\alpha\textrm{-H\"ol},T}$ :
 \begin{displaymath}
 \|x\|_{\alpha\textrm{-H\"ol},T} :=
 \sup_{(s,t)\in\Delta_T}\frac{\|x_t - x_s\|}{|t - s|^{\alpha}}
 \textrm{ $;$ }
 \forall x\in C^{\alpha\textrm{-H\"ol}}([0,T],\mathbb R^d).
 \end{displaymath}
 \item The step-$N$ signature of $x\in C^{1\textrm{-H\"ol}}([0,T],\mathbb R^d)$ with $N\in\mathbb N^*$ is denoted by $S_N(x)$ :
 \begin{displaymath}
 S_N(x)_t :=
 \left(1,\int_{0 < u < t}dx_u,\dots,\int_{0 < u_1 <\dots < u_N < t}dx_{u_1}\otimes
 \dots\otimes dx_{u_N}\right)
 \textrm{ $;$ }
 \forall t\in [0,T].
 \end{displaymath}
 \item The step-$N$ free nilpotent group over $\mathbb R^d$ is denoted by $G^N(\mathbb R^d)$ :
 \begin{displaymath}
 G^N(\mathbb R^d) :=
 \{S_N(\gamma)_1
 \textrm{ $;$ }
 \gamma\in C^{1\textrm{-H\"ol}}([0,1],\mathbb R^d)\}.
 \end{displaymath}
 \item The space of the geometric $\alpha$-rough paths from $[0,T]$ into $G^{[1/\alpha]}(\mathbb R^d)$ is denoted by $G\Omega_{\alpha,T}(\mathbb R^d)$ :
 \begin{displaymath}
 G\Omega_{\alpha,T}(\mathbb R^d) :=
 \overline{\{S_{[1/\alpha]}(x)
 \textrm{ $;$ }x\in C^{1\textrm{-H\"ol}}([0,T],\mathbb R^d)\}}^{d_{\alpha\textrm{-H\"ol},T}}
 \end{displaymath}
 where, $d_{\alpha\textrm{-H\"ol},T}$ is the $\alpha$-H\"older distance for the Carnot-Carath\'eodory metric.
\end{itemize}
%


%
\section{Preliminaries}
The purpose of this section is to provide the appropriate formulation of Equation (\ref{main_equation}) in the rough paths framework. At the end of the section, a convergence result for the Euler scheme associated to Equation (\ref{main_equation}) is stated, and a definition of invariant sets for rough differential equations is provided.
\\
\\
The definitions and propositions stated in the major part of this section come from Lyons and Qian \cite{LQ02}, Friz and Victoir \cite{FV10}, or Friz and Hairer \cite{FH14}.
\\
\\
First, the signal $w$ is $\alpha$-H\"older continuous with $\alpha\in (0,1]$. In addition, $w$ has to satisfy the following assumption.
%


%
\begin{assumption}\label{geometric_rough_path_w}
There exists $\mathbf w\in G\Omega_{\alpha,T}(\mathbb R^e)$ such that $\mathbf w^{(1)} = w$.
\end{assumption}
\noindent
Let $W : [0,T]\rightarrow\mathbb R^{e + 1}$ be the signal defined by :
\begin{displaymath}
W_t :=
te_1 +
\sum_{k = 2}^{e + 1}
w_{t}^{(k - 1)}e_k
\textrm{ $;$ }
\forall t\in [0,T].
\end{displaymath}
By Friz and Victoir \cite{FV10}, Theorem 9.26, there exists at least one $\mathbb W\in G\Omega_{\alpha,T}(\mathbb R^{e + 1})$ such that $\mathbb W^{(1)} = W$.
\\
\\
Let us state the conditions the collection of vector fields of a rough differential equation has to satisfy in order to get at least the existence of solutions.
\\
\\
\textbf{Notation.} For every $\gamma > 0$, $\lfloor\gamma\rfloor$ is the largest integer strictly smaller than $\gamma$.
%


%
\begin{definition}\label{gamma_Lipschitz_maps}
Consider $\gamma > 0$, $l,m\in\mathbb N^*$ and a nonempty closed set $V\subset\mathbb R^l$. A map $f :\mathbb R^l\rightarrow\mathcal M_{l,m}(\mathbb R)$ is $\gamma$-Lipschitz continuous (in the sense of Stein) from $V$ into $\mathcal M_{l,m}(\mathbb R)$ if and only if :
\begin{enumerate}
 \item $f|_{V}\in C^{\lfloor\gamma\rfloor}(V,\mathcal M_{l,m}(\mathbb R))$.
 \item $f,Df,\dots,D^{\lfloor\gamma\rfloor}f$ are bounded on $V$.
 \item $D^{\lfloor\gamma\rfloor}f$ is $(\gamma -\lfloor\gamma\rfloor)$-H\"older continuous from $\mathbb R^l$ into $\mathcal L(V^{\otimes\lfloor\gamma\rfloor},\mathcal M_{l,m}(\mathbb R))$ (i.e. there exists $C > 0$ such that for every $x,y\in V$,
 \begin{displaymath}
 \|D^{\lfloor\gamma\rfloor}f(y) -
 D^{\lfloor\gamma\rfloor}f(x)\|_{\mathcal L(V^{\otimes\lfloor\gamma\rfloor},\mathcal M_{l,m}(\mathbb R))}
 \leqslant
 C\|y - x\|^{\gamma -\lfloor\gamma\rfloor}\textrm{).}
 \end{displaymath}
\end{enumerate}
The set of all such maps is denoted by $\normalfont{\textrm{Lip}}^{\gamma}(V,\mathcal M_{l,m}(\mathbb R))$.
\\
\\
The map $f$ is locally $\gamma$-Lipschitz continuous from $\mathbb R^l$ into $\mathcal M_{l,m}(\mathbb R)$, if for every nonempty compact set $K\subset\mathbb R^l$, $f$ is $\gamma$-Lipschitz continuous from $K$ into $\mathcal M_{l,m}(\mathbb R)$. The set of all such maps is denoted by $\normalfont{\textrm{Lip}}_{\normalfont{\textrm{loc}}}^{\gamma}(\mathbb R^l,\mathcal M_{l,m}(\mathbb R))$.
\end{definition}
\noindent
In the sequel, $b$ and $\sigma$ satisfy the following assumption.
%


%
\begin{assumption}\label{non_explosion_b_sigma}
There exists $\gamma\in (1/\alpha,[1/\alpha] + 1)$ such that :
\begin{enumerate}
 \item $b\in\normalfont{\textrm{Lip}}_{\normalfont{\textrm{loc}}}^{\gamma - 1}(\mathbb R^d)$ and $\sigma\in\normalfont{\textrm{Lip}}_{\normalfont{\textrm{loc}}}^{\gamma - 1}(\mathbb R^d,\mathcal M_{d,e}(\mathbb R))$.
 \item $b$ (resp. $\sigma$) is Lipschitz continuous from $\mathbb R^d$ into itself (resp. $\mathcal M_{d,e}(\mathbb R)$).
 \item $D^{[1/\alpha]}b$ (resp. $D^{[1/\alpha]}\sigma$) is $(\gamma - [1/\alpha])$-H\"older continuous from $\mathbb R^d$ into $\mathcal L((\mathbb R^d)^{\otimes [1/\alpha]},\mathbb R^d)$ (resp. $\mathcal L((\mathbb R^d)^{\otimes [1/\alpha]},\mathcal M_{d,e}(\mathbb R))$).
\end{enumerate}
\end{assumption}
\noindent
Let $f_{b,\sigma} :\mathbb R^d\rightarrow\mathcal M_{d,e + 1}(\mathbb R)$ be the map defined by :
\begin{displaymath}
f_{b,\sigma}(x) :=
\sum_{k = 1}^{d}
b^{(k)}(x)e_{k,1} +
\sum_{l = 2}^{e + 1}
\sum_{k = 1}^{d}
\sigma_{k,l}(x)e_{k,l}
\textrm{ $;$ }
\forall x\in\mathbb R^d.
\end{displaymath}
In the rough paths framework, $dy = f_{b,\sigma}(y)d\mathbb W$ with $y_0\in\mathbb R^d$ as initial condition is the appropriate formulation of Equation (\ref{main_equation}).
\\
\\
Under the Assumptions \ref{geometric_rough_path_w} and \ref{non_explosion_b_sigma}, by Friz and Victoir Theorem 10.26, Exercice 10.55 and Exercice 10.56, the rough differential equation $dy = f_{b,\sigma}(y)d\mathbb W$ with $y_0\in\mathbb R^d$ as initial condition has at least one solution $y$ on $[0,T]$. Precisely, there exists a sequence $(W^n)_{n\in\mathbb N}$ of elements of $C^{1\textrm{-H\"ol}}([0,T],\mathbb R^{e + 1})$ such that
\begin{equation}\label{signal_approximation}
\lim_{n\rightarrow\infty}
d_{\alpha\textrm{-H\"ol},T}(S_{[1/\alpha]}(W^n),\mathbb W) = 0
\end{equation}
and
\begin{equation}\label{solution_approximation}
\lim_{n\rightarrow\infty}
\|y - y^n\|_{\infty,T} = 0
\end{equation}
where, for every $n\in\mathbb N$, $y^n$ is the solution on $[0,T]$ of the ordinary differential equation $dy^n = f_{b,\sigma}(y^n)dW^n$ with $y_0$ as initial condition.
\\
\\
Moreover, if $b$ and $\sigma$ satisfy the following assumption, the solution of Equation (\ref{main_equation}) is unique and denoted by $\pi_{f_{b,\sigma}}(0,y_0,\mathbb W)$.
%


%
\begin{assumption}\label{gamma_Lipschitz_b_sigma}
$b\in\normalfont{\textrm{Lip}}_{\normalfont{\textrm{loc}}}^{\gamma}(\mathbb R^d)$ and $\sigma\in\normalfont{\textrm{Lip}}_{\normalfont{\textrm{loc}}}^{\gamma}(\mathbb R^d,\mathcal M_{d,e}(\mathbb R))$.
\end{assumption}
\noindent
Let us now define the Euler scheme for Equation (\ref{main_equation}) and state a convergence result.
\\
\\
Let $D := (t_0,\dots,t_n)$ be a dissection of $[0,T]$ with $n\in\mathbb N^*$. The Euler scheme $\widehat y^n := (\widehat y_{t_0}^{n},\dots,\widehat y_{t_n}^{n})$ for Equation (\ref{main_equation}) along the dissection $D$ is defined by
\begin{displaymath}
\widehat y_{t_k}^{n} :=
\mathfrak E^{\mathbb W_{t_{k - 1},t_k}}\circ
\dots\circ\mathfrak E^{\mathbb W_{t_0,t_1}}y_0
\textrm{ $;$ }
\forall k\in\llbracket 1,n\rrbracket
\end{displaymath}
with
\begin{displaymath}
\mathfrak E^gx :=
x +\mathcal E_{f_{b,\sigma}}(x,g)
\end{displaymath}
and
\begin{displaymath}
\mathcal E_{f_{b,\sigma}}(x,g) :=
\sum_{k = 1}^{[1/\alpha]}
\sum_{i_1,\dots,i_k = 1}^{e + 1}
f_{b,\sigma,i_1}\dots f_{b,\sigma,i_k}I(x)g^{(k),i_1,\dots,i_k}
\end{displaymath}
for every $g\in G^{[1/\alpha]}(\mathbb R^{e + 1})$ and $x\in\mathbb R^d$, and where $I$ denotes the identity map from $\mathbb R^d$ into itself.
%


%
\begin{theorem}\label{Euler_scheme_convergence}
Let $D := (t_0,\dots,t_n)$ be a dissection of $[0,t]$ with $t\in [0,T]$ and $n\in\mathbb N^*$. Under the Assumptions \ref{non_explosion_b_sigma} and \ref{gamma_Lipschitz_b_sigma}, there exists a constant $C > 0$ depending only on $\alpha$, $\gamma$, $f_{b,\sigma}$ and $\|\mathbb W\|_{\alpha\normalfont{\textrm{-H\"ol}},T}$ such that
\begin{displaymath}
\|\pi_{f_{b,\sigma}}(0,y_0;\mathbb W)_t -\widehat y_{t}^{n}\|
\leqslant Ct|D|^{\theta - 1}
\end{displaymath}
where, $\theta := (\lfloor\gamma\rfloor + 1)\alpha > 1$ and $|D|$ is the mesh of $D$.
\end{theorem}
\noindent
See Friz and Victoir \cite{FV10}, Theorem 10.30.
\\
\\
Finally, let us state a definition of invariant sets for Equation (\ref{main_equation}).
\\
\\
Let $S$ be a subset of $\mathbb R^d$.
%


%
\begin{definition}\label{viability}
A function $\varphi : [0,T]\rightarrow\mathbb R^d$ is viable in $S$ if and only if,
\begin{displaymath}
\varphi(t)\in S
\textrm{ $;$ }
\forall t\in [0,T].
\end{displaymath}
\end{definition}
\noindent
The following definition provides a natural extension of the notion of invariant set in the rough paths theory setting.
%


%
\begin{definition}\label{invariant_sets}
\white .\black
\begin{enumerate}
 \item The subset $S$ is invariant for $(\sigma,\mathbf w)$ if and only if, for any initial condition $y_0\in S$, every solution on $[0,T]$ of the rough differential equation $dy =\sigma(y)d\mathbf w$ is viable in $S$.
 \item The subset $S$ is invariant for $(b,\sigma,\mathbb W)$ if and only if, for any initial condition $y_0\in S$, every solution on $[0,T]$ of the rough differential equation $dy = f_{b,\sigma}(y)d\mathbb W$ is viable in $S$.
\end{enumerate}
\end{definition}
%


%
\section{An invariance theorem for rough differential equations}
Consider a nonempty closed set $K\subset\mathbb R^d$. For every map $\varphi :\mathbb R^d\rightarrow\mathcal M_{d,m}(\mathbb R)$ with $m\in\mathbb N^*$, consider
\begin{align}\label{def-k-phi}
K_{\varphi} :=
\bigcap_{k = 1}^{m}
\{x\in\mathbb R^d :\forall x^*\in\Pi_K(x)\textrm{$,$ }
\varphi_{.,k}(x)
\in T_K(x^*)^{\circ\circ}\},
\end{align}
and then
\begin{align}\label{def-k-b-sigma}
K_{b,\pm\sigma} :=
K_b\cap K_{\sigma}\cap K_{-\sigma}.
\end{align}
The invariance of $K$ for $(b,\sigma,\mathbb W)$ is studied in this section under the two following assumptions on the maps $b$ and $\sigma$, and the signal $w$.
%


%
\begin{assumption}\label{viability_assumption}
$K\subset K_{b,\pm\sigma}$.
\end{assumption}
%


%
\begin{assumption}\label{signal_behavior}
There exists $\lambda,\mu\in ]0,\infty[$, $\beta\in (0,2\alpha\wedge 1)$, $t_0\in (0,T]$, $l\in\mathcal S_{t_0}$ and a countable set $\mathcal B_e\subset\partial B_e(0,1)$ such that $\{\pm e_k ; k\in\llbracket 1,e\rrbracket\}\subset\mathcal B_e$, $\overline{\mathcal B_e} =\partial B_e(0,1)$ and
\begin{displaymath}
-\mu =
\inf_{\delta\in\mathcal B_e}
\liminf_{t\rightarrow 0^+}
\frac{\langle \delta,w_t\rangle}{t^{\beta}l(t)}
\leqslant\sup_{\delta\in\mathcal B_e}
\liminf_{t\rightarrow 0^+}
\frac{\langle \delta,w_t\rangle}{t^{\beta}l(t)} = -\lambda.
\end{displaymath}
\end{assumption}
\noindent
Consider also the following stronger assumption on the maps $b$ and $\sigma$ :
%


%
\begin{assumption}\label{viability_assumption_weak}
$K_{b,\pm\sigma} =\mathbb R^d$.
\end{assumption}
\noindent
Now, let us state the main result of the paper ; the invariance theorem.
%


%
\begin{theorem}\label{viability_theorem}
Under the Assumptions \ref{geometric_rough_path_w} and \ref{non_explosion_b_sigma} on $b$ and $\sigma$ :
\begin{enumerate}
 \item Under Assumption \ref{viability_assumption_weak}, $K$ is invariant for $(b,\sigma,\mathbb W)$.
 \item When $K$ is convex :
 \begin{enumerate}
  \item Under Assumption \ref{viability_assumption}, $K$ is invariant for $(b,\sigma,\mathbb W)$.
  \item Under the Assumptions \ref{gamma_Lipschitz_b_sigma} and \ref{signal_behavior}, if $K$ is invariant for $(b,\sigma,\mathbb W)$, then Assumption \ref{viability_assumption} is fulfilled.
 \end{enumerate}
 \item When $K$ is compact and $b\equiv 0$, under Assumption \ref{signal_behavior}, if $K$ is invariant for $(\sigma,\mathbf w)$, then Assumption \ref{viability_assumption} is fulfilled.
\end{enumerate}
\end{theorem}
%


%
\begin{remark}\label{viability_assumption_remarks}
\white .\black
\begin{enumerate}
 \item By Remark \ref{tangent_cone_interior}, for any map $\varphi :\mathbb R^d\rightarrow\mathcal M_{d,m}(\mathbb R)$ with $m\in\mathbb N^*$, $\normalfont{\textrm{int}}(K)\subset K_{\varphi}$. So, in particular, $\normalfont{\textrm{int}}(K)\subset K_{b,\pm\sigma}$. Therefore, Assumption \ref{viability_assumption} is satisfied if and only if $\partial K\subset K_{b,\pm\sigma}$.
 \item If $K$ is convex, then
 \begin{displaymath}
 K_{\varphi} =
 \bigcap_{k = 1}^{m}
 \{x\in\mathbb R^d :\varphi_{.,k}(x)\in T_K(p_K(x))\}
 \end{displaymath}
 for every $\varphi :\mathbb R^d\rightarrow\mathcal M_{d,m}(\mathbb R)$ with $m\in\mathbb N^*$, and where $p_K(x)$ is the unique projection of $x\in\mathbb R^d$ on $K$.
 \item As in Aubin and Da Prato \cite{AD90}, when $K$ is not convex, the sufficient condition involves all $x\in\mathbb R^d$, and not only all $x\in K$ (see the statement of Theorem 1.5 and its remark page 601).
 \item Assumption \ref{viability_assumption} for some usual convex subsets of $\mathbb R^d$ :
 \begin{itemize}
  \item When $K$ is a vector subspace of $\mathbb R^d$, Assumption \ref{viability_assumption} means that
  \begin{displaymath}
  b(K)\subset K
  \end{displaymath}
  and
  \begin{displaymath}
  \sigma_{.,k}(K)\subset K
  \textrm{ $;$ }
  \forall k\in\llbracket 1,e\rrbracket.
  \end{displaymath}
  \item When $K$ is the unit ball of $\mathbb R^d$, Assumption \ref{viability_assumption} means that for every $x\in\mathbb R^d$ such that $\|x\| = 1$,
  \begin{displaymath}
  \langle x,b(x)\rangle\leqslant 0
  \end{displaymath}
  and
  \begin{displaymath}
  \langle\sigma_{.,k}(x),x\rangle = 0
  \textrm{ $;$ }
  \forall k\in\llbracket 1,e\rrbracket.
  \end{displaymath}
  \item Consider the polyhedron
  \begin{displaymath}
  K =\bigcap_{i\in I}\{x\in\mathbb R^d :\langle s_i,x - a_i\rangle\leqslant 0\}
  \end{displaymath}
  where, $I\subset\mathbb N$ is a nonempty finite set, and $(a_i)_{i\in I}$ and $(s_i)_{i\in I}$ are two families of elements of $\mathbb R^d$ such that $s_i\not= 0$ for every $i\in I$. Here, Assumption \ref{viability_assumption} means that for every $x\in K$ and $i\in I$ such that $\langle s_i,x - a_i\rangle = 0$,
  \begin{displaymath}
  \langle s_i,b(x)\rangle\leqslant 0
  \end{displaymath}
  and
  \begin{displaymath}
  \langle s_i,\sigma_{.,k}(x)\rangle = 0
  \textrm{ $;$ }
  \forall k\in\llbracket 1,e\rrbracket.
  \end{displaymath}
  These conditions on $b$ and $\sigma$ are quite natural, and the same as in Milian \cite{MILIAN95} or Cresson et al. \cite{CPS13}, where the driving signal of the main equation is a Brownian motion.
 \end{itemize}
 \item Assumption \ref{signal_behavior} is close to the notion of "signal rough at time 0" stated at \cite{FH14}, Chapter 6.
 \item Almost all the paths of the $e$-dimensional fractional Brownian motion satisfy Assumption \ref{signal_behavior} (see Proposition \ref{fbm_behavior}).
\end{enumerate}
\end{remark}
\noindent
At Subsection 3.1, the invariance of $K$ for $(b,\sigma,\mathbb W)$ is proved under Assumption \ref{viability_assumption_weak}, and under Assumption \ref{viability_assumption} when $K$ is convex. At Subsection 3.2, when $K$ is compact and $b\equiv 0$, under Assumption \ref{signal_behavior}, the necessity of Assumption \ref{viability_assumption} to get the invariance of $K$ for $(\sigma,\mathbf w)$ is proved. At Subsection 3.3, when $K$ is convex, under the Assumptions \ref{gamma_Lipschitz_b_sigma} and \ref{signal_behavior}, the necessity of Assumption \ref{viability_assumption} to get the invariance of $K$ for $(b,\sigma,\mathbb W)$ is proved.
%


%
\subsection{Sufficient condition of invariance}
The main purpose of this subsection is to prove the invariance of $K$ for $(b,\sigma,\mathbb W)$ under Assumption \ref{viability_assumption_weak}, and under Assumption \ref{viability_assumption} when $K$ is convex (Theorem \ref{viability_theorem}.(1,2.a)). As in Aubin and Da Prato \cite{AD90}, the proof deeply relies on the fact that $t\in [0,T]\mapsto d_{K}^{2}(y_t)$ has a nonpositive epiderivative (see J.P. Aubin et al. \cite{ABSP11}, Section 18.6.2), where $y$ is the solution of $dy = f_{b,\sigma}(y)dW$ with $\alpha = 1$ and $y_0\in K$ as initial condition (see Lemma \ref{sufficient_viability_lemma}). Finally, when $K$ is convex and compact, Corollary \ref{sufficient_viability_corollary} allows to relax the regularity assumptions on $b$ and $\sigma$.
%


%
\begin{lemma}\label{sufficient_viability_lemma}
Let $K$ be nonempty closed subset of $\mathbb R^d$. Under the Assumptions \ref{non_explosion_b_sigma} and \ref{viability_assumption_weak} with $\alpha = 1$, the solution $y$ on $[0,T]$ of the ordinary differential equation $dy = f_{b,\sigma}(y)dW$ with $y_0\in K$ as initial condition is viable in $K$.
\end{lemma}
%


%
\begin{proof}
In order to show that $y$ is viable in $K$ in a second step, as in Aubin and Da Prato \cite{AD90}, the following inequality is proved in a first step :
\begin{equation}\label{epiderivative_bound}
\liminf_{h\rightarrow 0^+}
\frac{d_{K}^{2}(y_{t + h}) - d_{K}^{2}(y_t)}{h}
\leqslant 0.
\end{equation}
\textbf{Step 1.} For $t\in [0,T]$ and $h > 0$,
\begin{equation}\label{sufficient_viability_lemma_1}
y_{t + h} - y_t =
b(y_t)h +\sigma(y_t)(w_{t + h} - w_t) + R_{t,h}
\end{equation}
with
\begin{displaymath}
R_{t,h} :=
\int_{t}^{t + h}[b(y_s) - b(y_t)]ds +
\int_{t}^{t + h}[\sigma(y_s) -\sigma(y_t)]dw_s.
\end{displaymath}
Since $y$ (resp. $b$) is Lipschitz continuous from $[0,T]$ (resp. $\mathbb R^d$) into $\mathbb R^d$, there exists $C_1 > 0$ such that
\begin{displaymath}
\left\|\int_{t}^{t + h}[b(y_s) - b(y_t)]ds\right\|
\leqslant C_1h^2.
\end{displaymath}
The function $w$ is Lipschitz continuous from $[0,T]$ into $\mathbb R^e$, so there exists $\dot w\in L^{\infty}([0,T],\mathbb R^e)$ such that :
\begin{displaymath}
w_s = w_0 +
\int_{0}^{s}\dot w_udu
\textrm{ $;$ }
\forall s\in [0,T].
\end{displaymath}
Then,
\begin{eqnarray*}
 \left\|\int_{t}^{t + h}[\sigma(y_s) -\sigma(y_t)]dw_s\right\| & = &
 \left\|\int_{t}^{t + h}[\sigma(y_s) -\sigma(y_t)]\dot w_sds\right\|\\
 & \leqslant &
 \|\dot w\|_{\infty,T}\int_{t}^{t + h}\|\sigma(y_s) -\sigma(y_t)\|_{\mathcal M_{d,e}(\mathbb R)}ds.
\end{eqnarray*}
Since $y$ (resp. $\sigma$) is Lipschitz continuous from $[0,T]$ (resp. $\mathbb R^d$) into $\mathbb R^d$ (resp. $\mathcal M_{d,e}(\mathbb R)$), there exists $C_2 > 0$ such that :
\begin{displaymath}
\int_{t}^{t + h}\|\sigma(y_s) -\sigma(y_t)\|_{\mathcal M_{d,e}(\mathbb R)}ds
\leqslant C_2h^2.
\end{displaymath}
Therefore,
\begin{equation}\label{sufficient_viability_lemma_2}
\|R_{t,h}\|\leqslant C_3h^2
\end{equation}
with $C_3 := C_1 + C_2\|\dot w\|_{\infty,T}$.
\\
\\
For $y_{t}^{*}\in\Pi_K(y_t)$ and $y_{t + h}^{*}\in\Pi_K(y_{t + h})$ arbitrarily chosen :
\begin{eqnarray*}
 d_{K}^{2}(y_{t + h}) & = &
 \|y_{t + h} - y_{t + h}^{*}\|^2\\
 & \leqslant &
 \|y_{t + h} - y_{t}^{*}\|^2\\
 & = &
 \|y_{t + h} - y_t\|^2 +
 2\langle y_{t + h} - y_t,y_t - y_{t}^{*}\rangle +
 \|y_t - y_{t}^{*}\|^2.
\end{eqnarray*}
So,
\begin{equation}\label{sufficient_viability_lemma_3}
d_{K}^{2}(y_{t + h}) -
d_{K}^{2}(y_t)
\leqslant
\|y_{t + h} - y_t\|^2 +
2\langle y_{t + h} - y_t,y_t - y_{t}^{*}\rangle.
\end{equation}
The function $y$ is Lipschitz continuous from $[0,T]$ into $\mathbb R^d$, so there exists $C_4 > 0$ such that :
\begin{displaymath}
\|y_{t + h} - y_t\|^2\leqslant C_4h^2.
\end{displaymath}
By Equality (\ref{sufficient_viability_lemma_1}) and Inequality (\ref{sufficient_viability_lemma_2}) :
\begin{eqnarray*}
 \langle y_{t + h} - y_t,y_t - y_{t}^{*}\rangle\leqslant
 C_3h^2 & + & \\
 \langle b(y_t),y_t - y_{t}^{*}\rangle h & + &
 \sum_{k = 1}^{e}\langle\sigma_{.,k}(y_t),y_t - y_{t}^{*}\rangle (w_{t + h}^{(k)} - w_{t}^{(k)}).
\end{eqnarray*}
Moreover,
\begin{displaymath}
\langle b(y_t),y_t - y_{t}^{*}\rangle h +
\sum_{k = 1}^{e}\langle\sigma_{.,k}(y_t),y_t - y_{t}^{*}\rangle (w_{t + h}^{(k)} - w_{t}^{(k)})
\leqslant 0
\end{displaymath}
because $y_t - y_{t}^{*}\in T_K(y_{t}^{*})^{\circ}$ by Proposition \ref{orthogonal_projections_closed}, and $b(y_t),\pm\sigma_{.,k}(y_t)\in T_K(y_{t}^{*})^{\circ\circ}$ for every $k\in\llbracket 1,e\rrbracket$ by Assumption \ref{viability_assumption_weak}. So,
\begin{displaymath}
\langle y_{t + h} - y_t,y_t - y_{t}^{*}\rangle
\leqslant C_3h^2.
\end{displaymath}
Therefore, by Inequality (\ref{sufficient_viability_lemma_3}) :
\begin{displaymath}
d_{K}^{2}(y_{t + h}) -
d_{K}^{2}(y_t)
\leqslant
C_5h^2
\end{displaymath}
with $C_5 := 2C_3 + C_4$. This achieves the first step.
\\
\\
\textbf{Step 2.} Consider the function $\varphi : [0,T]\rightarrow\mathbb R_+$ defined by :
\begin{displaymath}
\varphi(t) :=
d_{K}^{2}(y_t)
\textrm{ $;$ }
\forall t\in [0,T].
\end{displaymath}
Assume that there exists $\tau\in (0,T]$ such that $\varphi(\tau) > 0$. Since $\varphi$ is continuous on $[0,T]$, the set
\begin{displaymath}
\{t\in [0,\tau) :\forall s\in (t,\tau]\textrm{$,$ }\varphi(s) > 0\}
\end{displaymath}
is not empty, and its infimum is denoted by $t_*$. Moreover, if $\varphi(t_*) > 0$, then there exists $t_{**}\in [0,t_*)$ such that $\varphi(t) > 0$ for every $t\in (t_{**},t_{*}]$. So, necessarily, $\varphi(t_*) = 0$.
\\
\\
By Inequality (\ref{epiderivative_bound}), for every $t\in [t_*,\tau]$,
\begin{displaymath}
D_{\uparrow}\varphi(t)(1) :=
\liminf_{h\rightarrow 0^+,u\rightarrow 1}
\frac{\varphi(t + hu) -\varphi(t)}{h}\leqslant 0.
\end{displaymath}
So, by Aubin \cite{AUBIN90} :
\begin{displaymath}
0 <\varphi(\tau) =
\varphi(\tau) -\varphi(t_*)\leqslant 0.
\end{displaymath}
There is a contradiction, then $\varphi$ is nonpositive on $[0,T]$. Since $\varphi([0,T])\subset\mathbb R_+$, necessarily :
\begin{displaymath}
\varphi(t) = d_{K}^{2}(y_t) = 0
\textrm{ $;$ }
\forall t\in [0,T].
\end{displaymath}
In other words, $y$ is viable in $K$.
\end{proof}
\noindent
Via Lemma \ref{sufficient_viability_lemma}, let us prove Theorem \ref{viability_theorem}.(1,2).
%


%
\begin{proof}
\textit{Theorem \ref{viability_theorem}.(1,2a)}. Theorem \ref{viability_theorem}.(1) is proved at the first step, and Theorem \ref{viability_theorem}.(2a) is proved at the second step.
\\
\\
\textbf{Step 1.} Assume that $\alpha\in (0,1]$ and $K_{b,\pm\sigma} =\mathbb R^d$. Since $K$ is a closed subset of $\mathbb R^d$, every solution on $[0,T]$ of the rough differential equation $dy = f_{b,\sigma}(y)d\mathbb W$ with $y_0\in K$ as initial condition is viable in $K$ by Lemma \ref{sufficient_viability_lemma} together with Equality (\ref{solution_approximation}).
\\
\\
\textbf{Step 2.} Assume that $\alpha = 1$, $K$ is convex and $K\subset K_{b,\pm\sigma}$ (see (\ref{def-k-b-sigma}) for a definition). Let $y$ be the solution on $[0,T]$ of the ordinary differential equation $dy = f_{b,\sigma}(y)dW$ with $y_0\in K$ as initial condition. Consider the maps $B := b\circ p_K$, $S :=\sigma\circ p_K$ and $F :\mathbb R^d\rightarrow\mathcal M_{d,e + 1}(\mathbb R)$ such that :
\begin{displaymath}
F(x) :=
\sum_{k = 1}^{d}
B^{(k)}(x)e_{k,1} +
\sum_{l = 2}^{e + 1}
\sum_{k = 1}^{d}
S_{k,l}(x)e_{k,l}
\textrm{ $;$ }
\forall x\in\mathbb R^d.
\end{displaymath}
Since $f_{b,\sigma}$ (resp. $p_K$) is Lipschitz continuous from $\mathbb R^d$ into $\mathcal M_{d,e + 1}(\mathbb R)$ (resp. $K$), $F$ is Lipschitz continuous from $\mathbb R^d$ into $\mathcal M_{d,e + 1}(\mathbb R)$. So, the ordinary differential equation $dY = F(Y)dW$ with $y_0\in K$ as initial condition has a unique solution $Y$ on $[0,T]$.
\\
\\
Since $K_{B,\pm S} =\mathbb R^d$, $Y$ is viable in $K$ by Lemma \ref{sufficient_viability_lemma}.
\\
\\
Therefore, the solution $y$ of the ordinary differential equation $dy = f_{b,\sigma}(y)dW$ with $y_0$ as initial condition coincides with $Y$ on $[0,T]$ because $f_{b,\sigma}$ coincides with $F$ on $K$. In particular, $y$ is viable in $K$.
\\
\\
Assume now that $\alpha\in (0,1]$. Since $K$ is a closed subset of $\mathbb R^d$, every solution on $[0,T]$ of the rough differential equation $dy = f_{b,\sigma}(y)d\mathbb W$ with $y_0\in K$ as initial condition is viable in $K$ by Equality (\ref{solution_approximation}).
\end{proof}
%


%
\begin{corollary}\label{sufficient_viability_corollary}
Under the Assumptions \ref{geometric_rough_path_w} and \ref{viability_assumption}, if $K$ is a nonempty convex and compact subset of $\mathbb R^d$, $b\in\normalfont{\textrm{Lip}}_{\normalfont{\textrm{loc}}}^{\gamma - 1}(\mathbb R^d)$ and $\sigma\in\normalfont{\textrm{Lip}}_{\normalfont{\textrm{loc}}}^{\gamma - 1}(\mathbb R^d,\mathcal M_{d,e}(\mathbb R))$ with $\gamma > 1/\alpha$, then all the solutions of the rough differential equation $dy = f_{b,\sigma}(y)d\mathbb W$ with $y_0\in K$ as initial condition are defined on $[0,T]$ and viable in $K$.
\end{corollary}
%


%
\begin{proof}
Since $f_{b,\sigma}$ is locally $(\gamma - 1)$-Lipschitz continuous from $\mathbb R^d$ into $\mathcal M_{d,e + 1}(\mathbb R)$, there exists $\tau\in (0,T]$ such that the rough differential equation $dy = f_{b,\sigma}(y)d\mathbb W$ with $y_0\in K$ as initial condition has at least one solution $y$ on $[0,\tau)$.
\\
\\
Since $b$ and $\sigma$ satisfy Assumption \ref{viability_assumption}, by Theorem \ref{viability_theorem}.(2) applied to $(b,\sigma,\mathbb W)$ on \mbox{$[0,\tau[$ :}
\begin{displaymath}
y_t\in K
\textrm{ $;$ }
\forall t\in [0,\tau[.
\end{displaymath}
So, $y$ is bounded on $[0,\tau[$ by at least one continuous function from $[0,T]$ into $\mathbb R^d$ because $K$ is a bounded subset of $\mathbb R^d$.
\\
\\
Therefore, by Friz and Victoir \cite{FV10}, Theorem 10.21, $y$ is defined on $[0,T]$, and by Theorem \ref{viability_theorem}.(2), it is viable in $K$.
\end{proof}
%


%
\subsection{Necessary condition of invariance : compact case}
When $K$ is compact and $b\equiv 0$, the purpose of this subsection is to prove that under Assumption \ref{signal_behavior}, if $K$ is invariant for $(\sigma,\mathbf w)$, then Assumption \ref{viability_assumption} is fulfilled (Theorem \ref{viability_theorem}.(3)).
%


%
\begin{lemma}\label{upper_bound_w}
Under Assumption \ref{signal_behavior}, for $T_0 := t_0\wedge T$,
\begin{displaymath}
M :=
\max_{k\in\llbracket 1,e\rrbracket}
\sup_{t\in [0,T_0]}
\left|\frac{w_{t}^{(k)}}{t^{\beta}l(t)}\right| <\infty.
\end{displaymath}
\end{lemma}
%


%
\begin{proof}
By Assumption \ref{signal_behavior} :
\begin{eqnarray*}
 \min_{k\in\llbracket 1,e\rrbracket}
 \liminf_{t\rightarrow 0^+}
 \frac{w_{t}^{(k)}}{t^{\beta}l(t)} & = &
 \min_{k\in\llbracket 1,e\rrbracket}
 \liminf_{t\rightarrow 0^+}
 \frac{\langle e_k,w_t\rangle}{t^{\beta}l(t)}\\
 & \geqslant &
 \inf_{\delta\in\mathcal B_e}
 \liminf_{t\rightarrow 0^+}
 \frac{\langle\delta,w_t\rangle}{t^{\beta}l(t)} =
 -\mu.
\end{eqnarray*}
Moreover, the function
\begin{displaymath}
t\in (0,T_0]
\longmapsto
\frac{w_{t}^{(k)}}{t^{\beta}l(t)}
\end{displaymath}
is continuous. So, there exists $r > 0$ such that :
\begin{displaymath}
\min_{k\in\llbracket 1,e\rrbracket}
\inf_{t\in [0,T_0]}
\frac{w_{t}^{(k)}}{t^{\beta}l(t)}
\geqslant -r\mu.
\end{displaymath}
Similarly, there exists $R > 0$ such that :
\begin{displaymath}
\max_{k\in\llbracket 1,e\rrbracket}
\sup_{t\in [0,T_0]}
\frac{w_{t}^{(k)}}{t^{\beta}l(t)}
\leqslant R\mu.
\end{displaymath}
This achieves the proof.
\end{proof}
%


%
\begin{proposition}\label{extension_signal_behavior}
Under Assumption \ref{signal_behavior} :
\begin{displaymath}
-\mu =
\inf_{\delta\in\partial B_e(0,1)}
\liminf_{t\rightarrow 0^+}
\frac{\langle\delta,w_t\rangle}{t^{\beta}l(t)}
\leqslant
\sup_{\delta\in\partial B_e(0,1)}
\liminf_{t\rightarrow 0^+}
\frac{\langle\delta,w_t\rangle}{t^{\beta}l(t)} = -\lambda.
\end{displaymath}
\end{proposition}
%


%
\begin{proof}
Let $\delta\in\partial B_e(0,1)$ be arbitrarily chosen. Since $\overline{\mathcal B_e} =\partial B_e(0,1)$, there exists a sequence $(\delta_n)_{n\in\mathbb N}$ of elements of $\mathcal B_e$ such that :
\begin{displaymath}
\lim_{n\rightarrow\infty}
\|\delta -\delta_n\| = 0.
\end{displaymath}
By Lemma \ref{upper_bound_w} :
\begin{displaymath}
M :=
\max_{k\in\llbracket 1,e\rrbracket}
\sup_{t\in [0,T_0]}
\left|\frac{w_{t}^{(k)}}{t^{\beta}l(t)}\right| <\infty.
\end{displaymath}
For every $n\in\mathbb N$ and $t\in (0,T_0]$,
\begin{eqnarray*}
 \frac{\langle\delta,w_t\rangle}{t^{\beta}l(t)}
 & = & \frac{\langle\delta -\delta_n,w_t\rangle}{t^{\beta}l(t)} +
 \frac{\langle\delta_n,w_t\rangle}{t^{\beta}l(t)}\\
 & \leqslant &
 \|\delta -\delta_n\|\cdot
 \frac{\|w_t\|}{t^{\beta}l(t)} +
 \frac{\langle\delta_n,w_t\rangle}{t^{\beta}l(t)}\\
 & \leqslant &
 \|\delta -\delta_n\|M +
 \frac{\langle\delta_n,w_t\rangle}{t^{\beta}l(t)}.
\end{eqnarray*}
So, for every $n\in\mathbb N$, by Assumption \ref{signal_behavior} :
\begin{eqnarray*}
 \liminf_{t\rightarrow 0^+}
 \frac{\langle\delta,w_t\rangle}{t^{\beta}l(t)}
 & \leqslant &
 \|\delta -\delta_n\|M +
 \liminf_{t\rightarrow 0^+}
 \frac{\langle\delta_n,w_t\rangle}{t^{\beta}l(t)}\\
 & \leqslant &
 \|\delta -\delta_n\|M -\lambda\\
 & &
 \xrightarrow[n\rightarrow\infty]{}
 -\lambda.
\end{eqnarray*}
Since the right hand side of the previous inequality is not depending on $\delta$ :
\begin{displaymath}
\sup_{\delta\in\partial B_e(0,1)}
\liminf_{t\rightarrow 0^+}
\frac{\langle\delta,w_t\rangle}{t^{\beta}l(t)}
\leqslant
-\lambda.
\end{displaymath}
Moreover, by Assumption \ref{signal_behavior} and since $\mathcal B_e\subset\partial B(0,1)$ :
\begin{displaymath}
-\lambda =
\sup_{\delta\in\mathcal B_e}
\liminf_{t\rightarrow 0^+}
\frac{\langle\delta,w_t\rangle}{t^{\beta}l(t)}
\leqslant
\sup_{\delta\in\partial B_e(0,1)}
\liminf_{t\rightarrow 0^+}
\frac{\langle\delta,w_t\rangle}{t^{\beta}l(t)}.
\end{displaymath}
Therefore,
\begin{displaymath}
\sup_{\delta\in\partial B_e(0,1)}
\liminf_{t\rightarrow 0^+}
\frac{\langle\delta,w_t\rangle}{t^{\beta}l(t)} =
-\lambda.
\end{displaymath}
Similarly,
\begin{displaymath}
\inf_{\delta\in\partial B_e(0,1)}
\liminf_{t\rightarrow 0^+}
\frac{\langle\delta,w_t\rangle}{t^{\beta}l(t)} = -\mu.
\end{displaymath}
\end{proof}
%


%
\begin{proof}
\textit{Theorem \ref{viability_theorem}.(3).} Let $y$ be a solution of the rough differential equation $dy =\sigma(y)d\mathbf w$ with $y_0\in K$ as initial condition, and assume that $y$ is viable in $K$.
\\
\\
Consider $s\in T_K(y_0)^{\circ}$ and $\varepsilon > 0$. Since $K$ is a nonempty compact subset of $\mathbb R^d$, by Proposition \ref{tangent_cone_compact_sets}, there exists $\delta_{\varepsilon} > 0$ such that for every $x\in K\cap B_d(y_0,\delta_{\varepsilon})$,
\begin{displaymath}
\langle x - y_0,s\rangle
\leqslant
\varepsilon\|x - y_0\|.
\end{displaymath}
Since $y$ is continuous, there exists $t_{\varepsilon}\in [0,T]$ such that $y([0,t_{\varepsilon}])\subset B_d(y_0,\delta_{\varepsilon})$. So, for every $t\in [0,t_{\varepsilon}]$,
\begin{displaymath}
\langle y_t - y_0,s\rangle
\leqslant
\varepsilon\|y_t - y_0\|.
\end{displaymath}
Then,
\begin{displaymath}
\limsup_{t\rightarrow 0^+}
\frac{\langle y_t - y_0,s\rangle}{t^{\beta}l(t)}
\leqslant
\varepsilon
\limsup_{t\rightarrow 0^+}
\frac{\|y_t - y_0\|}{t^{\beta}l(t)}.
\end{displaymath}
For every $t\in [0,T]$, by Theorem \ref{Euler_scheme_convergence} applied with the dissection $(0,t)$ of $[0,t]$, there exists a constant $C > 0$, depending on $T$ but not on $t$, such that :
\begin{eqnarray*}
 \|y_t - y_0 -\sigma(y_0)w_t\|
 & \leqslant &
 Ct|(0,t)|^{(\lfloor\gamma\rfloor + 1)\alpha - 1}\\
 & \leqslant &
 CT^{(\lfloor\gamma\rfloor - 1)\alpha}t^{2\alpha}.
\end{eqnarray*}
Since $\beta < 2\alpha$,
\begin{displaymath}
\limsup_{t\rightarrow 0^+}
\frac{\langle\sigma(y_0)w_t,s\rangle}{t^{\beta}l(t)}
\leqslant
\varepsilon
\limsup_{t\rightarrow 0^+}
\frac{\|\sigma(y_0)w_t\|}{t^{\beta}l(t)}.
\end{displaymath}
Therefore, by duality in $\mathbb R^d$ :
\begin{equation}\label{control_necessary_condition_1}
\liminf_{t\rightarrow 0^+}
\frac{\langle u(s),w_t\rangle}{t^{\beta}l(t)}
\geqslant 0
\textrm{ $;$ }
\forall s\in T_K(y_0)^{\circ}
\end{equation}
where, $u : T_K(y_0)^{\circ}\rightarrow\mathbb R^e$ is the map defined by
\begin{displaymath}
u(s) :=
-\sigma(y_0)^{\tau}s
\textrm{ $;$ }
\forall s\in T_K(y_0)^{\circ},
\end{displaymath}
and $\sigma(y_0)^{\tau}$ is the transpose of the matrix $\sigma(y_0)$.
\\
\\
Assume that there exists $s\in T_K(y_0)^{\circ}$ such that $u(s)\not= 0$, and put
\begin{displaymath}
v(s) :=\frac{u(s)}{\|u(s)\|}
\in\partial B_e(0,1).
\end{displaymath}
By Inequality (\ref{control_necessary_condition_1}) :
\begin{displaymath}
\liminf_{t\rightarrow 0^+}
\frac{\langle v(s),w_t\rangle}{t^{\beta}l(t)}
\geqslant 0.
\end{displaymath}
There is a contradiction with Assumption \ref{signal_behavior} by Proposition \ref{extension_signal_behavior}. So, necessarily :
\begin{displaymath}
u(s) = 0
\textrm{ $;$ }
\forall s\in T_K(y_0)^{\circ}.
\end{displaymath}
Therefore, since $(e_k)_{k\in\llbracket 1,e\rrbracket}$ is a basis of $\mathbb R^e$ :
\begin{displaymath}
\langle\sigma_{.,k}(y_0),s\rangle = 0
\textrm{ $;$ }
\forall k\in\llbracket 1,e\rrbracket.
\end{displaymath}
In particular, $\pm\sigma_{.,k}(y_0)\in T_K(y_0)^{\circ\circ}$ for every $k\in\llbracket 1,e\rrbracket$. This achieves the proof because $y_0\in K$ has been arbitrarily chosen.
\end{proof}
%


%
\subsection{Necessary condition of invariance : convex case}
When $K$ is convex, the purpose of this subsection is to prove that under the Assumptions \ref{gamma_Lipschitz_b_sigma} and \ref{signal_behavior}, if $K$ is invariant for $(b,\sigma,\mathbb W)$, then $b$ and $\sigma$ satisfy Assumption \ref{viability_assumption} (Theorem \ref{viability_theorem}.(2b)). First, that result is proved for the half-hyperplanes.
%


%
\begin{lemma}\label{necessary_viability_lemma}
Under the Assumptions \ref{geometric_rough_path_w}, \ref{non_explosion_b_sigma}, \ref{gamma_Lipschitz_b_sigma} and \ref{signal_behavior}, if there exists $\nu\in\llbracket 1,d\rrbracket$ such that $D_{\nu} =\{x\in\mathbb R^d : x^{(\nu)}\geqslant 0\}$ is invariant for $(b,\sigma,\mathbb W)$, then
\begin{displaymath}
b^{(\nu)}(x - x^{(\nu)}e_{\nu})\geqslant 0
\textrm{ and }\sigma_{\nu,.}(x - x^{(\nu)}e_{\nu}) = 0
\end{displaymath}
for every $x\in\mathbb R^d$.
\end{lemma}
%


%
\begin{proof}
Let $\widehat y : [0,T]\rightarrow\mathbb R^d$ be the map defined by :
\begin{displaymath}
\widehat y_t :=
\mathfrak E^{\mathbb W_{0,t}}y_0
\textrm{ $;$ }
\forall t\in [0,T].
\end{displaymath}
For every $t\in [0,T]$, $(y_0,\widehat y_t)$ coincides with the Euler scheme for the rough differential equation $dy = f_{b,\sigma}(y)d\mathbb W$ with $y_0\in\mathbb R^d$ as initial condition along the dissection $D_t := (0,t)$ of $[0,t]$.
\\
\\
In a first step, it is proved that if there exists $y_0\in D_{\nu}$ such that
\begin{equation}\label{solution_bound_RS}
\liminf_{t\rightarrow 0^+}\frac{\widehat y_{t}^{(\nu)}}{t^{\beta}l(t)} < 0,
\end{equation}
then $D_{\nu}$ is not invariant for $(b,\sigma,\mathbb W)$. In a second step, it is proved that if there exists $y_0\in\partial D_{\nu}$ such that $\sigma_{\nu,.}(y_0)\not= 0$ or $b^{(\nu)}(y_0) < 0$, then $\widehat y$ satisfies Inequality (\ref{solution_bound_RS}).
\\
\\
\textbf{Step 1.} Assume that there exists $y_0\in D_{\nu}$ such that :
\begin{displaymath}
\liminf_{t\rightarrow 0^+}
\frac{\widehat y_{t}^{(\nu)}}{t^{\beta}l(t)} < 0.
\end{displaymath}
For every $t\in [0,T]$, by Theorem \ref{Euler_scheme_convergence} applied with the dissection $(0,t)$ of $[0,t]$, there exists a constant $C_1 > 0$, depending on $T$ but not on $t$, such that
\begin{eqnarray*}
 \|\pi_{f_{b,\sigma}}(0,y_0;\mathbb W)_t -\widehat y_t\|
 & \leqslant & C_1t|D_t|^{\theta - 1}\\
 & = & C_1t^{\theta}
\end{eqnarray*}
with $\theta := (\lfloor\gamma\rfloor + 1)\alpha > 1$. So,
\begin{displaymath}
\frac{\pi_{f_{b,\sigma}}^{(\nu)}(0,y_0;\mathbb W)_t}{t^{\beta}l(t)}
\leqslant
C_1\frac{t^{\theta -\beta}}{l(t)} +
\frac{\widehat y_{t}^{(\nu)}}{t^{\beta}l(t)}
\textrm{ $;$ }
\forall t\in (0,T_0].
\end{displaymath}
Moreover, since $\theta > 1 >\beta$ and $l\in\mathcal S_{t_0}$ :
\begin{displaymath}
\liminf_{t\rightarrow 0^+}\left[
C_1\frac{t^{\theta -\beta}}{l(t)} +
\frac{\widehat y_{t}^{(\nu)}}{t^{\beta}l(t)}\right] =
\liminf_{t\rightarrow 0^+}
\frac{\widehat y_{t}^{(\nu)}}{t^{\beta}l(t)}.
\end{displaymath}
Therefore, by Inequality (\ref{solution_bound_RS}) :
\begin{displaymath}
\liminf_{t\rightarrow 0^+}
\frac{\pi_{f_{b,\sigma}}^{(\nu)}(0,y_0;\mathbb W)_t}{t^{\beta}l(t)} < 0.
\end{displaymath}
In conclusion, there exists $t_1\in [0,T_0]$ such that :
\begin{displaymath}
\pi_{f_{b,\sigma}}^{(\nu)}(0,y_0;\mathbb W)_{t_1} < 0.
\end{displaymath}
The path $\pi_{f_{b,\sigma}}^{(\nu)}(0,y_0;\mathbb W)$ is not viable in $D_{\nu}$.
\\
\\
\textbf{Step 2.} Let us show that if there exists $y_0\in\partial D_{\nu}$ such that $\sigma_{\nu,.}(y_0)\not= 0$ or $b^{(\nu)}(y_0) < 0$, then $\widehat y$ satisfies Inequality (\ref{solution_bound_RS}).
\\
\\
\textit{Case 1.} Assume that there exists $y_0\in\partial D_{\nu}$ such that $\sigma_{\nu,.}(y_0)\not= 0$. By \mbox{Lemma \ref{upper_bound_w} :}
\begin{displaymath}
M :=
\max_{k\in\llbracket 1,e\rrbracket}
\sup_{t\in [0,T_0]}
\left|\frac{w_{t}^{(k)}}{t^{\beta}l(t)}\right| <\infty.
\end{displaymath}
On the one hand, since $\|\sigma_{\nu,.}(y_0)\|^{-1}\sigma_{\nu,.}(y_0)\in\partial B_e(0,1)$ and $\overline{\mathcal B_e} =\partial B_e(0,1)$, there exists a sequence $(u_n)_{n\in\mathbb N}$ of elements of $\mathcal B_e$ such that :
\begin{displaymath}
\lim_{n\rightarrow\infty}
\left\|u_n -\frac{\sigma_{\nu,.}(y_0)}{\|\sigma_{\nu,.}(y_0)\|}\right\| = 0.
\end{displaymath}
So, there exists $n_0\in\mathbb N$ such that for every $n\in\mathbb N\cap [n_0,\infty[$,
\begin{displaymath}
\|\sigma_{\nu,.}(y_0) - v_n\|\leqslant
\frac{\lambda}{4eM}\|\sigma_{\nu,.}(y_0)\|
\end{displaymath}
and $v_n :=\|\sigma_{\nu,.}(y_0)\|u_n$ where, $\lambda$ is defined in Assumption \ref{signal_behavior}. Then, for every $n\in\mathbb N\cap [n_0,\infty[$,
\begin{eqnarray}
 \sup_{t\in [0,T_0]}
 \frac{1}{t^{\beta}l(t)}
 |\langle\sigma_{\nu,.}(y_0) - v_n,w_t\rangle|
 & \leqslant &
 eM\|\sigma_{\nu,.}(y_0) - v_n\|
 \nonumber\\
 \label{necessary_viability_lemma_2}
 & \leqslant &
 \frac{\lambda}{4}\|\sigma_{\nu,.}(y_0)\|.
\end{eqnarray}
On the other hand, by the definition of the Euler scheme for Equation (\ref{main_equation}), there exists $C_2 > 0$ such that :
\begin{displaymath}
|\mathcal E_{f_{b,\sigma}}^{(\nu)}(y_0,\mathbb W_{0,t}) -\langle\sigma_{\nu,.}(y_0),w_t\rangle|
\leqslant C_2(t^{2\alpha}\vee t)
\textrm{ $;$ }
\forall t\in [0,T].
\end{displaymath}
Since $l\in\mathcal S_{t_0}$ and $\beta\in (0,2\alpha\wedge 1)$ :
\begin{eqnarray}
 \lim_{t\rightarrow 0^+}
 \frac{1}{t^{\beta}l(t)}
 |\mathcal E_{f_{b,\sigma}}^{(\nu)}(y_0,\mathbb W_{0,t}) -\langle\sigma_{\nu,.}(y_0),w_t\rangle| & \leq  &
 C_2\lim_{t\rightarrow 0^+}\frac{t^{2\alpha -\beta}}{l(t)}
 \nonumber\\
 \label{necessary_viability_lemma_3}
 & = & 0.
\end{eqnarray}
For every $n\in\mathbb N\cap [n_0,\infty[$, by Assumption \ref{signal_behavior}, Inequality (\ref{necessary_viability_lemma_2}) and Equality (\ref{necessary_viability_lemma_3}) together :
\begin{eqnarray*}
 \liminf_{t\rightarrow 0^+}
 \frac{\widehat y_{t}^{(\nu)}}{t^{\beta}l(t)} & = &
 \liminf_{t\rightarrow 0^+}
 \frac{\langle\sigma_{\nu,.}(y_0),w_t\rangle}{t^{\beta}l(t)}\\
 & \leqslant &
 \sup_{t\in [0,T_0]}\frac{1}{t^{\beta}l(t)}|\langle\sigma_{\nu,.}(y_0) - v_n,w_t\rangle| +
 \liminf_{t\rightarrow 0^+}
 \frac{\langle v_n,w_t\rangle}{t^{\beta}l(t)}\\
 & \leqslant &
 \frac{\lambda}{4}\|\sigma_{\nu,.}(y_0)\|
 -\lambda\|\sigma_{\nu,.}(y_0)\|
 = -\frac{3\lambda}{4}\|\sigma_{\nu,.}(y_0)\|.
\end{eqnarray*}
So,
\begin{displaymath}
\liminf_{t\rightarrow 0^+}
 \frac{\widehat y_{t}^{(\nu)}}{t^{\beta}l(t)} < 0.
\end{displaymath}
By the first step of the proof, it means that $D_{\nu}$ is not invariant for $(b,\sigma,\mathbb W)$. In other words, if $K$ is invariant for $(b,\sigma,\mathbb W)$, then
\begin{equation}\label{viability_assumption_sigma}
\sigma_{\nu,.}(x) = 0
\textrm{ $;$ }
\forall x\in\partial D_{\nu}.
\end{equation}
\textit{Case 2.} Assume that $D_{\nu}$ is invariant for $(b,\sigma,\mathbb W)$ and there exists $y_0\in\partial D_{\nu}$ such that $b^{(\nu)}(y_0) < 0$. By the first case, since $D_{\nu}$ is invariant for $(b,\sigma,\mathbb W)$, Equation (\ref{viability_assumption_sigma}) is true.
\\
\\
Let $t\in [0,T]$ be arbitrarily chosen.
\begin{itemize}
 \item If $\alpha\in (1/2,1)$, then
 \begin{eqnarray*}
  \widehat y_{t}^{(\nu)} & = & b^{(\nu)}(y_0)t +\langle\sigma_{\nu,.}(y_0),w_t\rangle\\
  & = &
  b^{(\nu)}(y_0)t.
 \end{eqnarray*}
 \item If $\alpha\in (0,1/2]$, then
 \begin{displaymath}
 \widehat y_{t}^{(\nu)} =
 b^{(\nu)}(y_0)t +
 \langle\sigma_{\nu,.}(y_0),w_t\rangle +
 \sum_{k = 2}^{[1/\alpha]}\sum_{i_1,\dots,i_k = 1}^{e}
 [\sigma_{i_1}\dots\sigma_{i_k}I(y_0)]^{(\nu)}\mathbf w^{(k),i_1,\dots,i_k}.
 \end{displaymath}
 Consider $k\in\llbracket 2,[1/\alpha]\rrbracket$ and $i_1,\dots,i_k\in\llbracket 1,e\rrbracket$. There exists a real family $(\rho_{i_2,\dots,i_{k - 1},l_1,\dots,l_{k - 1}}(y_0))_{(l_1,\dots,l_{k - 1})\in\llbracket 1,d\rrbracket^{k - 1}}$ such that :
 \begin{displaymath}
 [\sigma_{i_1}\dots\sigma_{i_k}I(y_0)]^{(\nu)} =
 \sum_{l_1,\dots,l_{k - 1} = 1}^{d}
 \sigma_{l_1,i_1}(y_0)\partial_{l_1,\dots,l_{k - 1}}^{k - 1}\sigma_{\nu,i_k}(y_0)
 \rho_{i_2,\dots,i_{k - 1},l_1,\dots,l_{k - 1}}(y_0).
 \end{displaymath}
 Consider the set
 \begin{displaymath}
 I_{k - 1} :=
 \{(l_1,\dots,l_{k - 1})\in\llbracket 1,d\rrbracket^{k - 1} :
 \exists\kappa\in\llbracket 1,k - 1\rrbracket\textrm{$,$ }
 l_{\kappa}\not=\nu\}.
 \end{displaymath}
 By (\ref{viability_assumption_sigma}) together with Schwarz's lemma :
 \begin{displaymath}
 \partial_{l_1,\dots,l_{k - 1}}^{k - 1}\sigma_{\nu,.}(x) = 0
 \textrm{ $;$ }
 \forall x\in\partial D_{\nu}
 \textrm{$,$ }
 \forall (l_1,\dots,l_{k - 1})\in I_{k - 1}.
 \end{displaymath}
 Then,
 \begin{small}
 \begin{eqnarray*}
  [\sigma_{i_1}\dots\sigma_{i_k}I(y_0)]^{(\nu)} & = &
  \sigma_{\nu,i_1}(y_0)\partial_{\nu,\dots,\nu}^{k - 1}\sigma_{\nu,i_k}(y_0)
  \rho_{i_2,\dots,i_{k - 1},\nu,\dots,\nu}(y_0) +\\
  & &
  \sum_{l_1,\dots,l_{k - 1}\in I_{k  - 1}}
  \sigma_{l_1,i_1}(y_0)\partial_{l_1,\dots,l_{k - 1}}^{k - 1}\sigma_{\nu,i_k}(y_0)
  \rho_{i_2,\dots,i_{k - 1},l_1,\dots,l_{k - 1}}(y_0)\\
  & = & 0.
 \end{eqnarray*}
 \end{small}
 \newline
 So,
 \begin{displaymath}
 \widehat y_{t}^{(\nu)} =
 b^{(\nu)}(y_0)t.
 \end{displaymath}
\end{itemize}
Since $l\in\mathcal S_{t_0}$ and $1 >\beta$, there exists $t_2\in (0,T_0]$ such that :
\begin{displaymath}
0 <\frac{t^{1-\beta}}{l(t)} < 1
\textrm{ $;$ }
\forall t\in (0,t_2].
\end{displaymath}
Therefore,
\begin{eqnarray*}
 \liminf_{t\rightarrow 0^+}
 \frac{\widehat y_{t}^{(\nu)}}{t^{\beta}l(t)}
 & \leqslant &
 \liminf_{t\rightarrow 0^+}
 \frac{\widehat y_{t}^{(\nu)}}{t}\\
 & \leqslant &
 b^{(\nu)}(y_0) < 0.
\end{eqnarray*}
This achieves the proof because there is a contradiction by the first step of the proof.
\end{proof}
\noindent
Via Lemma \ref{necessary_viability_lemma}, let us prove that Assumption \ref{viability_assumption} is necessary to get the invariance of $K$ for $(b,\sigma,\mathbb W)$.
%


%
\begin{proof}
\textit{Theorem \ref{viability_theorem}.(2b).} In a first step, it is proved that if the half-space
\begin{displaymath}
H_{a,s} :=\{x\in\mathbb R^d :\langle s,x - a\rangle\leqslant 0\}
\end{displaymath}
with $a\in\mathbb R^d$ and $s\in\mathbb R^d\backslash\{0\}$ is invariant for $(b,\sigma,\mathbb W)$, then
\begin{displaymath}
\langle b(x),s\rangle\leqslant 0
\end{displaymath}
and
\begin{displaymath}
\langle\sigma_{.,k}(x),s\rangle = 0
\textrm{ $;$ }
\forall k\in\llbracket 1,e\rrbracket
\end{displaymath}
for every $x\in\partial H_{a,s}$. In a second step, this result is used to show that if $K$ is invariant for $(b,\sigma,\mathbb W)$, then $K\subset K_{b,\pm\sigma}$.
\\
\\
\textbf{Step 1.} Let $y_0\in H_{a,s}$ be arbitrarily chosen. Since $s\in\mathbb R^d\backslash\{0\}$, there exists $\nu\in\llbracket 1,d\rrbracket$ such that $s^{(\nu)}\not= 0$. Consider the map $U :\mathbb R^d\rightarrow\mathbb R^d$ defined by :
\begin{displaymath}
U(x) :=
-(x - a) +
(x^{(\nu)} - a^{(\nu)} -\langle x - a,s\rangle)e_{\nu}
\textrm{ $;$ }
\forall x\in\mathbb R^d.
\end{displaymath}
The map $U$ is one to one from $\mathbb R^d$ into itself, and
\begin{displaymath}
U^{-1}(x) =
-x + a -\frac{1}{s^{(\nu)}}
(x^{(\nu)} -\langle x,s\rangle)e_{\nu}
\textrm{ $;$ }
\forall x\in\mathbb R^d.
\end{displaymath}
Moreover, $U|_{H_{a,s}}$ (resp. $U|_{\partial H_{a,s}}$) is one to one from $H_{a,s}$ (resp. $\partial H_{a,s}$) into $D_{\nu}$ (resp. $\partial D_{\nu}$) where, $D_{\nu}$ is defined in Lemma \ref{necessary_viability_lemma}. For every $x,h\in\mathbb R^d$,
\begin{eqnarray*}
 DU(x).h & = &
 -h + (h^{(\nu)} -\langle h,s\rangle)e_{\nu}\\
 & = &
 M_Uh
\end{eqnarray*}
with
\begin{displaymath}
M_U :=
-I + e_{\nu,\nu} -
\sum_{k = 1}^{d}
s^{(k)}e_{\nu,k}.
\end{displaymath}
Consider the maps $B :\mathbb R^d\rightarrow\mathbb R^d$ and $S :\mathbb R^d\rightarrow\mathcal M_{d,e}(\mathbb R)$ defined by
\begin{eqnarray*}
 B(x) & := &
 M_Ub(U^{-1}(x))\\
 & = &
 (e_{\nu,\nu} - I)b(U^{-1}(x)) -
 \langle b(U^{-1}(x)),s\rangle e_{\nu}
\end{eqnarray*}
and
\begin{eqnarray*}
 S(x) & := &
 M_U\sigma(U^{-1}(x))\\
 & = &
 (e_{\nu,\nu} - I)\sigma(U^{-1}(x)) -
 \sum_{k = 1}^{e}
 \langle\sigma_{.,k}(U^{-1}(x)),s\rangle e_{\nu,k}
\end{eqnarray*}
for every $x\in\mathbb R^d$. Let $F :\mathbb R^d\rightarrow\mathcal M_{d,e + 1}(\mathbb R)$ be the map defined by :
\begin{displaymath}
F(x) :=
\sum_{k = 1}^{d}B^{(k)}(x)e_{k,1} +
\sum_{l = 2}^{e + 1}\sum_{k = 1}^{d}
S_{k,l}(x)e_{k,l}
\textrm{ $;$ }
\forall x\in\mathbb R^d.
\end{displaymath}
Since $U^{-1}\in\mathcal L(\mathbb R^d)$ and $b$ and $\sigma$ satisfy assumptions \ref{non_explosion_b_sigma} and \ref{gamma_Lipschitz_b_sigma}, $B$ and $S$ also. So, by Friz and Victoir \cite{FV10}, Theorem 10.26, Exercice 10.55 and Exercice 10.56, the rough differential equation $dz = F(z)d\mathbb W$ with $U(y_0)$ as initial condition has a unique solution $z$ on $[0,T]$. By the (rough) change of variable formula, for every $t\in [0,T]$,
\begin{eqnarray*}
 U^{-1}(z_t) & = &
 y_0 +
 \int_{0}^{t}
 M_{U}^{-1}dz_s\\
 & = &
 y_0 +
 \int_{0}^{t}
 f_{b,\sigma}(U^{-1}(z_s))dW_s.
\end{eqnarray*}
Therefore, since $d_y = f_{b,\sigma}(y) dW$ has a unique solution, 
\begin{equation}\label{necessary_viability_theorem_1}
\pi_F(0,.;\mathbb W) =
U(\pi_{f_{b,\sigma}}(0,U^{-1}(.);\mathbb W)).
\end{equation}
Assume that $H_{a,s}$ is invariant for $(b,\sigma,\mathbb W)$. Since $U|_{H_{a,s}}$ is one to one from $H_{a,s}$ into $D_{\nu}$, by Equality (\ref{necessary_viability_theorem_1}), $D_{\nu}$ is invariant for $(B,S,\mathbb W)$. So, by Lemma \ref{necessary_viability_lemma} :
\begin{displaymath}
B^{(\nu)}(x - x^{(\nu)}e_{\nu})\geqslant 0
\textrm{ and }
S_{\nu,.}(x - x^{(\nu)}e_{\nu}) = 0
\end{displaymath}
for every $x\in\mathbb R^d$. Let $k\in\llbracket 1,e\rrbracket$ and $x\in\partial H_{a,s}$ be arbitrarily chosen. Since $U|_{\partial H_{a,s}}$ is one to one from $\partial H_{a,s}$ into $\partial D_{\nu}$, $U(x)\in\partial D_{\nu}$. Therefore, by construction of $B$ and $S$ :
\begin{displaymath}
\langle b(x),s\rangle =
-B^{(\nu)}(U(x))\leqslant 0
\end{displaymath}
and
\begin{eqnarray*}
 \langle\sigma_{.,k}(x),s\rangle & = &
 -\sum_{j = 1}^{d}\sigma_{j,k}(x)s^{(j)}\\
 & = &
 -S_{\nu,k}(U(x)) = 0.
\end{eqnarray*}
\textbf{Step 2.} Assume that there exists $y_0\in\partial K$ such that $y_0\not\in K_{b,\pm\sigma}$. Then, there exists $s\in N_K(y_0)$ such that :
\begin{equation}\label{necessary_viability_theorem_2}
\langle b(y_0),s\rangle > 0
\textrm{ or }
(\exists k\in\llbracket 1,e\rrbracket :
\langle\sigma_{.,k}(y_0),s\rangle\not= 0).
\end{equation}
Consider the half-space
\begin{displaymath}
H_{y_0,s} :=\{x\in\mathbb R^d :\langle s,x - y_0\rangle\leqslant 0\}.
\end{displaymath}
By the first step of the proof, (\ref{necessary_viability_theorem_2}) implies that there exists $t\in [0,T]$ such that $\pi_{f_{b,\sigma}}(0,y_0;\mathbb W)_t\not\in H_{y_0,s}$. Moreover, since $y_0\in\partial K$ and $s\in N_K(y_0)$, $K\subset H_{y_0,s}$. Therefore, $\pi_{f_{b,\sigma}}(0,y_0;\mathbb W)_t\not\in K$. This achieves the proof by contraposition.
\end{proof}
%


%
\section{A comparison theorem for rough differential equations}
In this section, a comparison theorem for rough differential equations is proved by using the viability theorem of Section 3.
\\
\\
Consider a nonempty set $I\subset\llbracket 1,d\rrbracket$, and
\begin{displaymath}
K :=\{(x_1,x_2)\in (\mathbb R^d)^2 :\forall i\in I\textrm{, }x_{1}^{(i)}\leqslant x_{2}^{(i)}\}.
\end{displaymath}
The following comparison theorem is a consequence of Theorem \ref{viability_theorem}.
%


%
\begin{proposition}\label{comparison_theorem}
For $j\in\{1,2\}$, consider $b_j :\mathbb R^d\rightarrow\mathbb R^d$ and $\sigma^j :\mathbb R^d\rightarrow\mathcal M_{d,e}(\mathbb R)$ satisfying assumptions \ref{non_explosion_b_sigma} and \ref{gamma_Lipschitz_b_sigma}, and the map $f^j :\mathbb R^d\rightarrow\mathcal M_{d,e + 1}(\mathbb R)$ defined by :
\begin{displaymath}
f^j(x) :=
\sum_{k = 1}^{d}
b_{j}^{(k)}(x)e_{k,1} +
\sum_{l = 2}^{e + 1}
\sum_{k = 1}^{d}
\sigma_{k,l}^{j}(x)e_{k,l}
\textrm{ $;$ }
\forall x\in\mathbb R^d.
\end{displaymath}
Under the Assumptions \ref{geometric_rough_path_w} and \ref{signal_behavior}, the two following conditions are equivalent :
\begin{enumerate}
 \item For every $(y_{0}^{1},y_{0}^{2})\in K$, $i\in I$ and $t\in [0,T]$, $(y_{t}^{1})^{(i)}\leqslant (y_{t}^{2})^{(i)}$ where $y^j$ is the solution of the rough differential equation $dy^j = f^j(y^j)d\mathbb W$ with $y_{0}^{j}$ as initial condition for $j\in\{1,2\}$.
 \item For every $(x_1,x_2)\in K$ and $i\in I$, if $x_{1}^{(i)} = x_{2}^{(i)}$, then
 \begin{displaymath}
 b_{1}^{(i)}(x_1)\leqslant b_{2}^{(i)}(x_2)
 \end{displaymath}
 and
 \begin{displaymath}
 \sigma_{i,k}^{1}(x_1) =
 \sigma_{i,k}^{2}(x_2)
 \textrm{ $;$ }
 \forall k\in\llbracket 1,e\rrbracket.
 \end{displaymath}
\end{enumerate}
\end{proposition}
%


%
\begin{proof}
The set $K$ is (isomorphe to) a nonempty closed convex polyhedron of $\mathbb R^{2d}$. Indeed,
\begin{displaymath}
K\cong\bigcap_{i\in I}\{x\in\mathbb R^{2d} :\langle s_i,x\rangle\leqslant 0\}
\end{displaymath}
with $s_i := e_i - e_{d + i}$ for every $i\in I$. Let $F :\mathbb R^{2d}\rightarrow\mathcal M_{2d,e + 1}(\mathbb R)$ be the map defined by :
\begin{displaymath}
F(x_1,x_2) :=
\sum_{l = 1}^{e + 1}
\sum_{k = 1}^{d}
[f_{k,l}^{1}(x_1)e_{k,l} +
f_{k,l}^{2}(x_2)e_{d + k,l}]
\textrm{ $;$ }
\forall (x_1,x_2)\in(\mathbb R^d)^2.
\end{displaymath}
Since $b_j$ and $\sigma^j$ satisfy Assumption \ref{non_explosion_b_sigma} for $j\in\{1,2\}$, $B := F_{.,1}$ and $S := (F_{.,k})_{k\in\llbracket 2,e + 1\rrbracket}$ also.
\\
\\
The first condition is equivalent to the invariance of $K$ for $(B,S,\mathbb W)$, and the second condition means that $K\subset K_{B,\pm S}$ (see Milian \cite{MILIAN95}, Theorem 2 (proof)). Therefore, these conditions are equivalent by Theorem \ref{viability_theorem}.
\end{proof}
%


%
\section{Invariance for differential equations driven by a fractional Brownian motion}
At Subsection 4.1, it is shown that the fractional Brownian motion satisfies assumptions \ref{geometric_rough_path_w} and \ref{signal_behavior}. Subsection 4.2 deals with the viability of the solutions of a multidimensional logistic equation driven by a fractional Brownian motion of Hurst parameter belonging to $(1/4,1)$.
%


%
\subsection{Fractional Brownian motion}
In this subsection, it is proved that the fractional Brownian motion satisfies assumptions \ref{geometric_rough_path_w} and \ref{signal_behavior}. So, the viability theorem proved at Section 3 (Theorem \ref{viability_theorem}) can be applied to differential equations driven by a fractional Brownian motion. In particular, it extends the results of Aubin and Da Prato \cite{AD90}.
\\
\\
First of all, let us remind the definition of the fractional Brownian motion.
%


%
\begin{definition}\label{fBm}
Let $(\textrm B_t)_{t\in [0,T]}$ be an $e$-dimensional centered Gaussian process. It is a fractional Brownian motion of Hurst parameter $H\in (0,1)$ if and only if,
\begin{displaymath}
\normalfont{\textrm{cov}}(B_{s}^{(i)},B_{t}^{(j)}) =
\frac{1}{2}(|t|^{2H} + |s|^{2H} - |t - s|^{2H})\delta_{i,j}
\end{displaymath}
for every $(i,j)\in\llbracket 1,e\rrbracket^2$ and $(s,t)\in [0,T]^2$.
\end{definition}
\noindent
About the fractional Brownian motion, the reader can refer to Nualart \cite{NUALART06}, Chapter 5.
\\
\\
Let $B := (B_t)_{t\in [0,T]}$ be an $e$-dimensional fractional Brownian motion of Hurst parameter $H\in (1/4,1)$. The associated canonical probability space is denoted by $(\Omega,\mathcal A,\mathbb P)$.
\\
\\
By Garcia-Rodemich-Rumsey's lemma (see Nualart \cite{NUALART06}, Lemma A.3.1), almost all the paths of $B$ are $\alpha$-H\"older continuous with $\alpha\in (0,H)$. The following proposition ensures that almost all the paths of $B$ satisfy also Assumption \ref{signal_behavior}.
%


%
\begin{proposition}\label{fbm_behavior}
For any countable set $\mathcal B_e\subset\partial B_e(0,1)$, almost surely,
\begin{displaymath}
\liminf_{t\rightarrow 0^+}
\frac{\langle x,B_t\rangle}{t^{\beta}l(t)} = -1
\textrm{ $;$ }
\forall x\in\mathcal B_e
\end{displaymath}
with $\beta = H$ and $l\in\mathcal S_{e^{-1}}$ defined by
\begin{displaymath}
l(t) :=
\sqrt{2\log\left(\log\left(\frac{1}{t}\right)\right)}
\textrm{ $;$ }
\forall t\in (0,e^{-1}].
\end{displaymath}
\end{proposition}
%


%
\begin{proof}
By the law of the iterated logarithm for the 1-dimensional fractional Brownian motion (see Arcones \cite{ARCONES95} and Viitasaari \cite{VIITASAARI14}, Remark 2.3.3) :
\begin{displaymath}
\mathbb P\left[
\liminf_{t\rightarrow 0^+}
\frac{B_{t}^{(1)}}{t^{\beta}l(t)} = -1\right] = 1.
\end{displaymath}
Consider $x\in\partial B_e(0,1)$. Since $(\langle x,B_t\rangle)_{t\in [0,T]}\stackrel{d}{=}B^{(1)}$ :
\begin{displaymath}
\mathbb P\left[
\liminf_{t\rightarrow 0^+}
\frac{\langle x,B_t\rangle}{t^{\beta}l(t)} = -1\right] = 1.
\end{displaymath}
Therefore, since $\mathcal B_e$ is a countable subset of $\partial B_e(0,1)$ :
\begin{displaymath}
\mathbb P\left[
\bigcap_{x\in\mathcal B_e}
\left\{
\liminf_{t\rightarrow 0^+}
\frac{\langle x,B_t\rangle}{t^{\beta}l(t)} = -1
\right\}\right] = 1.
\end{displaymath}
This achieves the proof.
\end{proof}
\noindent
By Friz and Victoir \cite{FV10}, Proposition 15.5 and Theorem 15.33, there exists an enhanced Gaussian process $\mathbf B : (\Omega,\mathcal A)\rightarrow G\Omega_{\alpha,T}(\mathbb R^e)$ such that $\mathbf B^{(1)} = B$. So, the signal $B$ satisfies assumptions \ref{geometric_rough_path_w} and \ref{signal_behavior}.
\\
\\
Let $W := (W_t)_{t\in [0,T]}$ be the stochastic process defined by :
\begin{displaymath}
W_t :=
te_1 +\sum_{k = 1}^{e}B_{t}^{(k)}e_{k + 1}
\textrm{ $;$ }
\forall t\in [0,T].
\end{displaymath}
By Friz and Victoir \cite{FV10}, Theorem 9.26, there exists an enhanced stochastic process $\mathbb W : (\Omega,\mathcal A)\rightarrow G\Omega_{\alpha,T}(\mathbb R^{e + 1})$ such that $\mathbb W^{(1)} := W$.
\\
\\
Consider $\alpha\in (0,H)$ and a nonempty closed set $K\subset\mathbb R^d$.
%


%
\begin{proposition}\label{viability_fbm_theorem}
Under the Assumptions \ref{non_explosion_b_sigma} and \ref{gamma_Lipschitz_b_sigma} on $b$ and $\sigma$ :
\begin{enumerate}
 \item Under Assumption \ref{viability_assumption_weak}, $K$ is invariant for $(b,\sigma,\mathbb W)$.
 \item When $K$ is compact and $b\equiv 0$, if $K$ is invariant for $(\sigma,\mathbf B)$, then Assumption \ref{viability_assumption} is fulfilled.
 \item When $K$ is convex, it is invariant for $(b,\sigma,\mathbb W)$ if and only if Assumption \ref{viability_assumption} is fulfilled.
\end{enumerate}
\end{proposition}
%


%
\begin{proof}
Straightforward application of Theorem \ref{viability_theorem}.
\end{proof}
%


%
\begin{proposition}\label{viability_fbm_corollary}
Under Assumption \ref{viability_assumption}, if $K$ is convex and compact, $b\in\normalfont{\textrm{Lip}}_{\normalfont{\textrm{loc}}}^{\gamma - 1}(\mathbb R^d)$ and $\sigma\in\normalfont{\textrm{Lip}}_{\normalfont{\textrm{loc}}}^{\gamma - 1}(\mathbb R^d,\mathcal M_{d,e}(\mathbb R))$ with $\gamma > 1/\alpha$, then all the solutions of the rough differential equation $dY = f_{b,\sigma}(Y)d\mathbb W$ with $y_0\in K$ as initial condition are defined on $[0,T]$ and viable in $K$.
\end{proposition}
%


%
\begin{proof}
Straightforward application of Corollary \ref{sufficient_viability_corollary}.
\end{proof}
%


%
\begin{remark}\label{classic_SDE}
\white .\black
\begin{enumerate}
 \item The Brownian motion is a fractional Brownian motion of Hurst parameter $H = 1/2$.
 \item The rough differential equations driven by a Brownian motion are stochastic differential equations in the sense of Stratonovich. Let $B$ be an $e$-dimensional Brownian motion. In order to consider the stochastic differential equation
 \begin{displaymath}
 dy_t = b(y_t)dt +\sigma(y_t)dB_t
 \end{displaymath}
 in the sense of It\^o, one has to consider the rough differential equation
 \begin{displaymath}
 dy_t =\left[b(y_t) -
 \frac{1}{2}\sum_{i,j = 1}^{e}\sigma_{.,i}\sigma_{.,j}(y_t)\right]dt +
 \sigma(y_t)dB_t
 \end{displaymath}
 where,
 \begin{displaymath}
 \sigma_{.,i}\sigma_{.,j} :=
 \sum_{k = 1}^{d}\sigma_{k,i}\partial_k \sigma_{.,j}
 \textrm{ $;$ }
 \forall i,j\in\llbracket 1,e\rrbracket.
 \end{displaymath}
 (see Friz and Victoir \cite{FV10}, p. 510, Equation (17.3)).
\end{enumerate}
\end{remark}


%
\subsection{A logistic equation driven by a fractional Brownian motion}
The logistic equation is a typical example of differential equation with a non-Lipschitz vector field, but with solutions viable in a nonempty convex and compact subset of $\mathbb R^d$.
\\
\\
Consider $K := [0,1]^d$, $\gamma > 1/\alpha$, a locally $(\gamma - 1)$-Lipschitz continuous map $\sigma :\mathbb R^d\rightarrow\mathcal M_{d,e}(\mathbb R)$ such that $K\subset K_{\sigma}\cap K_{-\sigma}$, $m\in\mathbb R^d$ and $b_m :\mathbb R^d\rightarrow\mathbb R^d$ the map defined by :
\begin{displaymath}
b_{m}^{(i)}(x) :=
m^{(i)}x^{(i)}(1 - x^{(i)})
\textrm{ $;$ }
\forall i\in\llbracket 1,d\rrbracket
\textrm{$,$ }
\forall x\in\mathbb R^d.
\end{displaymath}
The set $K$ is a nonempty compact convex polyhedron of $\mathbb R^d$. Indeed,
\begin{displaymath}
K = K_1\cap K_2
\end{displaymath}
with
\begin{displaymath}
K_1 :=
\bigcap_{i = 1}^{d}\{x\in\mathbb R^d :\langle -e_i,x\rangle\leqslant 0\}
\end{displaymath}
and
\begin{displaymath}
K_2 :=
\bigcap_{i = 1}^{d}\{x\in\mathbb R^d :\langle e_i,x - e_i\rangle\leqslant 0\}.
\end{displaymath}
Consider $i\in\llbracket 1,d\rrbracket$ and $x\in\mathbb R^d$. If $\langle -e_i,x\rangle = 0$, then $x^{(i)} = 0$ and $\langle b_m(x),-e_i\rangle = 0$. If $\langle e_i,x - e_i\rangle = 0$, then $x^{(i)} = 1$ and $\langle b_m(x),e_i\rangle = 0$. Therefore, $K\subset K_{b_m}$. In other words, $b_m$ and $\sigma$ satisfy Assumption \ref{viability_assumption}.
\\
\\
Consider $y_0\in K$. Since $K$ is convex and compact, by Proposition \ref{viability_fbm_corollary}, the logistic equation
\begin{displaymath}
Y_t = y_0 +\int_{0}^{t}b_m(Y_s)ds +\int_{0}^{t}\sigma(Y_s)dB_s
\textrm{ $;$ }
t\in [0,T]
\end{displaymath}
has at least one solution defined on $[0,T]$ and viable in $K$. For instance, one can put
\begin{displaymath}
\sigma(x) :=
\sum_{i = 1}^{d}x^{(i)}(1 - x^{(i)})e_{i,i}
\textrm{ $;$ }
\forall x\in\mathbb R^d.
\end{displaymath}
\appendix
%


%
\section{Tangent and normal cones}
This Appendix is a brief survey on convex analysis.
\\
\\
The definitions and propositions stated in this subsection come from Hiriart-Urrut and Lemar\'echal \cite{HUL01}, Chapter A, and Aubin et al. \cite{ABSP11}, Chapter 18.
\\
\\
First, let us define the polar and bipolar sets of a closed cone.
%


%
\begin{definition}\label{polar_bipolar_cone}
\white .\black
\begin{enumerate}
 \item The polar set of a closed cone $K\subset\mathbb R^d$ is the closed cone
 \begin{displaymath}
 K^{\circ} =
 \{s\in\mathbb R^d :
 \forall\delta\in K
 \textrm{$,$ }
 \langle s,\delta\rangle\leqslant 0\}.
 \end{displaymath}
 \item The bipolar set of $K$ is the closed cone $K^{\circ\circ} := (K^{\circ})^{\circ}$.
\end{enumerate}
\end{definition}
\noindent
Let us now define the tangent and normal cones to a nonempty closed set $S\subset\mathbb R^d$ at $x\in S$.
%


%
\begin{definition}\label{tangent_normal_cone}
\white .\black
\begin{enumerate}
 \item A vector $\delta\in\mathbb R^d$ is tangent to $S$ at $x$ if and only if there exists a sequence $(x_n)_{n\in\mathbb N}$ of elements of $S$, and a real sequence $(t_n)_{n\in\mathbb N}$ such that when $n\rightarrow\infty$,
 \begin{displaymath}
 \|x_n - x\|\rightarrow 0
 \textrm{$,$ }
 t_n\downarrow 0
 \textrm{ and }
 \left\|
 \frac{x_n - x}{t_n} -\delta\right\|
 \rightarrow 0.
 \end{displaymath}
 The set of the tangent vectors to $S$ at $x$ is a closed cone of $\mathbb R^d$, called the tangent cone to $S$ at $x$, and denoted by $T_S(x)$.
 \item A vector $s\in\mathbb R^d$ is normal to $S$ at $x$ if and only if,
 \begin{displaymath}
 \langle s,\delta\rangle\leqslant 0
 \textrm{ $;$ }
 \forall\delta\in T_S(x).
 \end{displaymath}
 The set of the normal vectors to $S$ at $x$ is the normal cone to $S$ at $x$, denoted by $N_S(x)$.
\end{enumerate}
\end{definition}
%


%
\begin{remark}\label{tangent_cone_interior}
If $x\in\normalfont{\textrm{int}}(S)$, then $T_S(x) = T_S(x)^{\circ\circ} =\mathbb R^d$.
\end{remark}
\noindent
The two following properties are crucial in the proof of Theorem \ref{viability_theorem}.
%


%
\begin{proposition}\label{orthogonal_projections_closed}
For every $y\in\mathbb R^d$ and $y^*\in\Pi_S(y)$ (see (\ref{set_of_projections}) for a definition), $y - y^*\in T_S(y^*)^{\circ}$.
\end{proposition}
\noindent
See \cite {AUBIN91}, Proposition 3.2.3 p. 85.
%


%
%
%
%
%
%
%
%
%
%


%
\begin{proposition}\label{tangent_cone_compact_sets}
If $S$ is compact, then
\begin{displaymath}
T_S(x)^{\circ} =
\{s\in\mathbb R^d :
\forall\varepsilon > 0\textrm{$,$ }
\exists\delta > 0\textrm{$,$ }
\forall y\in S\cap B_d(x,\delta)\textrm{$,$ }
\langle y - x,s\rangle
\leqslant
\varepsilon\|y - x\|\}.
\end{displaymath}
\end{proposition}
\noindent
The two last propositions provide some properties of the tangent and normal cones when $S$ is a nonempty closed convex set.
%


%
\begin{proposition}\label{tangent_cone_properties}
The tangent cone $T_S(x)$ is a closed convex cone such that $S\subset\{x\} + T_S(x)$.
\end{proposition}
%


%
\begin{proposition}\label{normal_cone_convex}
A vector $s\in\mathbb R^d$ is normal to $S$ at $x$ if and only if,
\begin{displaymath}
\langle s,y - x\rangle
\leqslant 0
\textrm{ $;$ }
\forall y\in S.
\end{displaymath}
\end{proposition}
%


%

%

\begin{thebibliography}{99}
 \bibitem{ARCONES95} M.A. Arcones. \textit{On the Law of the Iterated Logarithm for Gaussian processes.} Journal of Theoretical Probability 8, 877-904,1995.
 \bibitem{AUBIN90} J.P. Aubin. \textit{A Survey of Viability Theory.} SIAM J. Control and Optimization 28, 4, 749-788, 1990.
 \bibitem{AUBIN91} J.P. Aubin. \textit{Viability Theory.} Modern Birkhauser Classics, 1991.
 \bibitem{ABSP11} J.P. Aubin, A. M. Bayen and P. Saint-Pierre. \textit{Viability Theory : New Directions. 2nd Edition.} Springer, 2011.
 \bibitem{AD90} J.P. Aubin and G. Da Prato. \textit{Stochastic Viability and Invariance.} Annali Scuola Normale di Pisa 27, 595-694, 1990.
 \bibitem{AD98} J.P. Aubin and G. Da Prato. \textit{The Viability Theorem for Stochastic Differential Inclusions.} Stochastic Anal. Appl. 16, 1, 1-15, 1998.
 \bibitem{BQRR02} R. Buckdahn, M. Quincampoix, C. Rainer and A. Rascanu. \textit{Viability of Moving Sets for Stochastic Differential Equation.} Advances in Differential Equations 7, 9, 1045-1072, 2002.
 \bibitem{BQRT10} R. Buckdahn, M. Quincampoix, C. Rainer and J. Teichmann. \textit{Another Proof for the Equivalence Between Invariance of Closed Sets with Respect to Stochastic and Deterministic Systems.} Bull. Sci. Math. 134, 207-214, 2010.
 \bibitem{CR08} I. Ciotir and A. Rascanu. \textit{Viability for Differential Equations driven by Fractional Brownian Motion.} Journal of Differential Equations 247, 5, 1505-1528, 2009.
 \bibitem{CPS13} J. Cresson, B. Puig and S. Sonner. \textit{Validating Stochastic Models : Invariance Criteria for Systems of Stochastic Differential Equations and the Selection of a Stochastic Hodgkin-Huxley Type Model.} Internat. J. Biomath. Biostat. 2, 111-122, 2013.
 \bibitem{CQ01} L. Coutin and Z. Qian. \textit{Stochastic Analysis, Rough Path Analysis and Fractional Brownian Motions.} Probab. Theory Related Fields 122, 1, 108-140, 2002.
 \bibitem{PF01} G. Da Prato and H. Frankowska. \textit{Invariance of Stochastic Control Systems with Deterministic Arguements.} J. Differential Equations 200, 18-52, 2004.
 \bibitem{FH14} P. Friz and M. Hairer. \textit{A Course on Rough Paths, With an Introduction to Regularity Structures.} Springer, 2014.
 \bibitem{FV10} P. Friz and N. Victoir. \textit{Multidimensional Stochastic Processes as Rough Paths : Theory and Applications.} Cambridge Studies in Applied Mathematics 120, Cambridge University Press, 2010.
 \bibitem{GT93} S. Gautier and L. Thibault. \textit{Viability for Constrained Stochastic Differential Equations.} Differential and Integral Equations 6, 6, 1395-1414, 1993.
 \bibitem{GUBINELLI04} M. Gubinelli. \textit{Controling Rough Paths.} J. Funct. Anal. 216, 1, 86-140, 2004.
 \bibitem{HUL01} J.B. Hiriart-Urrut and C. Lemar\'echal. \textit{Fundamentals of Convex Analysis.} Springer, 2001.
 \bibitem{LEJAY10} A. Lejay. \textit{Controlled Differential Equations as Young Integrals : A Simple Approach.} Journal of Differential Equations 248, 1777-1798, 2010.
 \bibitem{LS84} P.L. Lions and A.S. Sznitman. \textit{Stochastic Differential Equations with Reflecting Boundary Conditions.} Communications on Pure and Applied Mathematics XXXVII, 511-537, 1984.
\bibitem{LYONS98} T. Lyons. \textit{Differential Equations Driven by Rough Signals.} Rev. Mat. Iberoamericana 14, 2, 215-310, 1998.
 \bibitem{LQ02} T. Lyons and Z. Qian. \textit{System Control and Rough Paths.} Oxford University Press, 2002.
 \bibitem{MMS15} A. Melnikov, Y. Mishura and G. Shevchenko. \textit{Stochastic Viability and Comparison Theorems for Mixed Stochastic Differential Equations.} Methodology and Computing in Applied Probability 17, 1, 169-188, 2015.
 \bibitem{MICHTA98} M. Michta. \textit{A Note on Viability Under Distribution Constraints.} Discuss. Math. Algebra. Stoch. Methods 18, 2, 215-225, 1998.
 \bibitem{MILIAN93} A. Milian. \textit{A Note on Stochastic Invariance for It\^o Equations.} Bull. Pol. Acad. Sci. Math 41, 2, 1993.
 \bibitem{MILIAN95} A. Milian. \textit{Stochastic Viability and a Comparison Theorem.} Colloquium Mathematicum LXVIII, 2, 297-316, 1995.
 \bibitem{NAGUMO42} M. Nagumo. \textit{Uber die Lage der Integralkurven Gew\"onlicher Differentialgleichungen.} Proc. Phys. Math. Soc. Japan 24, 551Ð559, 1942.
 \bibitem{NR11} T. Nie and A. Rascanu. \textit{Deterministic Characterization of Viability for Stochastic Differential Equation Driven by Fractional Brownian Motion.} ESAIM:COCV 18, 4, 915-929, 2011.
 \bibitem{NUALART06} D. Nualart. \textit{The Malliavin Calculus and Related Topics. 2nd Edition.} Springer, 2006.
 \bibitem{NR02} D. Nualart and A. Rascanu. \textit{Differential Equations Driven by Fractional Brownian Motion.} Collect. Math. 53, 1, 55-81, 2002.
 \bibitem{VIITASAARI14} L. Viitasaari. \textit{Integration in a Normal World: Fractional Brownian Motion and Beyond.} Aalto university publication series, 2014.
\end{thebibliography}
\end{document}